\documentclass[10pt]{amsart}

\usepackage{amssymb}
\usepackage[leqno]{amsmath}
\usepackage{mathrsfs}
\usepackage{stmaryrd}
\usepackage{chemarrow}
\usepackage[norelsize]{algorithm2e}
\usepackage{enumerate}
\usepackage{graphicx}
\usepackage[pdftex,bookmarksnumbered,bookmarksopen,colorlinks,linkcolor=red,anchorcolor=black,citecolor=blue,urlcolor=blue]{hyperref}
\usepackage[all]{xy}
\usepackage{tikz}
\usetikzlibrary{arrows}

\usepackage{multirow}

  \newcounter{mnote}
  \let\oldmarginpar\marginpar
    \renewcommand\marginpar[1]{\-\oldmarginpar[\raggedleft\footnotesize #1]%
    {\raggedright\footnotesize #1}}

\newtheorem{theorem}{Theorem}[section]
\newtheorem{lemma}[theorem]{Lemma}
\newtheorem{corollary}[theorem]{Corollary}

\newtheorem{remark}[theorem]{Remark}

\newcommand{\bs}{\boldsymbol}

\renewcommand{\div}{\operatorname{div}}

\numberwithin{equation}{section}

\begin{document}
\title[$H^m$-nonconforming VEM]{Nonconforming Virtual Element Method for $2m$-th Order Partial Differential Equations in $\mathbb R^n$
 }
\author{Long Chen}%
\address{Department of Mathematics, University of California at Irvine, Irvine, CA 92697, USA}%
\email{chenlong@math.uci.edu}%
\author{Xuehai Huang}%
\address{School of Mathematics, Shanghai University of Finance and Economics, Shanghai 200433, China}%
\email{huang.xuehai@sufe.edu.cn}%

\thanks{The second author was supported by the National Natural Science Foundation of China Project 11771338, the Fundamental Research Funds for the Central
Universities 2019110066, and Zhejiang Provincial
Natural Science Foundation of China Project LY17A010010.}


\begin{abstract}
A unified construction of the $H^m$-nonconforming virtual elements of any order $k$ is developed on any shape of polytope in $\mathbb R^n$ with constraints $m\leq n$ and $k\geq m$.
As a vital tool in the construction, a generalized Green's identity for $H^m$ inner product is derived.
The $H^m$-nonconforming virtual element methods are then used to approximate
solutions of the $m$-harmonic equation.
After establishing a bound on the jump related to the weak continuity, the optimal error estimate of the canonical interpolation, and the norm equivalence of the stabilization term,
the optimal error estimates are derived for the $H^m$-nonconforming virtual element methods.
\end{abstract}
\maketitle


\section{Introduction}

We intend to construct $H^m$-nonconforming virtual elements of order $k\in\mathbb N$ on a very general polytope $K\subset\mathbb R^n$ in any dimension and any order with constraints $m\leq n$ and $k\geq m$.
Since an $m$th order derivative of polynomial degree $m-1$ or less would be zero,
the constraint $k\geq m$ is required in constructing $H^m$-nonconforming or conforming virtual elements to ensure that the virtual element spaces possess meaningful approximation in $H^m$-seminorm. 
Due to a technical reason, we will restrict to the case $m\leq n$ in this paper and postpone the case $m>n$ in future works.
The virtual element was described as a generalization of the finite element on a general polytope in~\cite{BeiraoBrezziCangianiManziniEtAl2013,BeiraoBrezziMariniRusso2014}, thus it is helpful to recall the definition of the finite element first.

A finite element on $K$ was defined as a triple $(K, V_K, \mathcal N_K)$ in~\cite{Ciarlet1978}, where $V_K$ is the finite-dimensional space of shape functions, and $\mathcal N_K$ the set of degrees of freedom (d.o.f.). The set $\mathcal N_K$ forms a basis of $(V_K)'$ the dual space of the space of shape functions.
The shape functions of the finite element are usually polynomials, and their basis functions being dual to the degrees of freedom $\mathcal N_K$ have to be explicitly constructed for the implementation, which is painful for high order cases (either $k,m,$ or $n$ is large).

We can also represent the virtual element as a triple $(K, \mathcal N_K, V_K)$. Here we reorder $V_K$ and $\mathcal N_K$ to emphasize that the set of the degrees of freedom $\mathcal N_K$ is crucial in the construction of the virtual element, and the space of shape functions $V_K$ is virtual.
Indeed after having the degrees of freedom $\mathcal N_K$, we may attach different spaces. The space of shape functions $V_K$ is only required to include all
polynomials of the total degree up to $k$ for the approximation property. Different from the finite element, one advantage of the virtual element is that the basis functions of $V_k$ are not explicitly required in the implementation. When forming the linear system of the virtual element method, the computation of all the to-be-required quantities can be transferred to the computation using the degrees of freedom.

Construction of $H^m$-conforming or nonconforming elements is an active topic in the field of the finite element methods in recent years.
Some $H^m$-conforming finite elements with polynomial shape functions were designed on the simplices in~\cite{ArgyrisFriedScharpf1968,Zenisek1970,BrambleZlamal1970,AlfeldSirvent1991,Zhang2009a} and on the hyperrectangles in~\cite{Zhang2010,HuZhang2015a,HuHuangZhang2011}.
Recently an $H^m$-conforming virtual element for polyharmonic problems with arbitrary $m$ in two dimensions was introduced and
studied in \cite{AntoniettiManziniVerani2018a}.
For general $m$, nonconforming elements on the simplices are easier to construct than conforming ones.
In~\cite{WangXu2013,WangXu2006}, Wang and Xu constructed the minimal $H^m$-nonconforming elements in any dimension with constraint $m\leq n$. Recently Wu and Xu extended these minimal $H^m$-nonconforming elements to $m=n+1$ by enriching the space of shape functions with bubble functions in~\cite{WuXu2019}, and to arbitrary $m$ and $n$ by using the interior penalty technique in~\cite{WuXu2017}. In two dimensions, Hu and Zhang designed the $H^m$-nonconforming elements on the triangle for any $m$ in~\cite{HuZhang2017}.
On the other hand, the $H^2$-conforming virtual element, the $C^{0}$-type $H^2$-nonconforming virtual element and the fully $H^2$-nonconforming virtual element on the polygon with any shape in two dimensions were developed in ~\cite{BrezziMarini2013},~\cite{ZhaoChenZhang2016} and ~\cite{AntoniettiManziniVerani2018,ZhaoZhangChenMao2018}, respectively.
In~\cite{Wang2019}, a nonconforming Crouzeix-Raviart type, i.e, $H^1$-nonconforming finite element was advanced on the polygon.

Although the $H^m$-conforming virtual element has been devised for $n=2$ in \cite{AntoniettiManziniVerani2018a} for arbitrary $m$, generalization to dimension $n>2$ seems nontrivial. While the $H^m$-nonconforming virtual element can be constructed in a universal way for all $n\geq m$ and allows unified error analysis.

We shall construct the $H^m$-nonconforming virtual element in any order on the polytope with any shape in any dimension (with constraints $k\geq m$ and $m\leq n$). The vital tool is the following generalized Green's identity for the $H^m$ space
\begin{align}
(\nabla^mu, \nabla^mv)_K = &((-\Delta)^mu, v)_K \notag\\
&+ \sum_{j=1}^m\sum_{F\in\mathcal F^j(K)}\sum_{\alpha\in A_{j}\atop|\alpha|\leq m-j}\Big (D^{2m-j-|\alpha|}_{F, \alpha}(u), \frac{\partial^{|\alpha|}v}{\partial\nu_{F}^{\alpha}}\Big )_F, \label{eq:HmGreenIntro}
\end{align}
which is proved by the mathematical induction and integration by parts.
Here $\mathcal F^j(K)$ is the set of all $(n-j)$-dimensional faces
of the polytope $K$, $A_{j}$ the set consisting of all $n$-dimensional multi-indexes $\alpha=(\alpha_1, \cdots, \alpha_n)$ with $\alpha_{j+1}=\cdots=\alpha_n=0$,  $D^{2m-j-|\alpha|}_{F, \alpha}(u)$ some $(2m-j-|\alpha|)$-th order derivatives of $u$ on $F$, and $\frac{\partial^{|\alpha|}v}{\partial\nu_{F}^{\alpha}}$ the multi-indexed normal derivatives on $F$.

Imagining $u$ in the Green's identity~\eqref{eq:HmGreenIntro} as a polynomial of degree $k$ temporarily, we acquire the degrees of freedom $\mathcal N_k(K)$ from the right hand side of the Green's identity~\eqref{eq:HmGreenIntro}. And the space $V_k(K)$ of shape functions is defined inherently by requiring the first terms in the inner product to be in polynomial spaces. Namely the right hand side of~\eqref{eq:HmGreenIntro} provides a natural duality of $V_k(K)$ and $\mathcal N_k(K)$. As a result we construct the fully $H^m$-nonconforming virtual element $(K, \mathcal N_k(K), V_k(K))$ completely based on the Green's identity~\eqref{eq:HmGreenIntro}.
If $K$ is a simplex and $k=m$, the virtual element $(K, \mathcal N_k(K), V_k(K))$ is reduced to the nonconforming finite element in~\cite{WangXu2013}, hence we generalize the nonconforming finite element in~\cite{WangXu2013} to high order $k>m$ and arbitrary polytopes.
In two dimensions, we also recover the fully $H^2$-nonconforming virtual element in~\cite{AntoniettiManziniVerani2018,ZhaoZhangChenMao2018}.

After introducing the local $H^m$ projection $\Pi^K$ and a stabilization term using d.o.f., we propose $H^m$-nonconforming virtual element methods for solving the $m$-harmonic equation. 
We assume the mesh $\mathcal T_h$ admits a virtual quasi-uniform triangulation, and each element in $\mathcal T_h$ is star-shaped.
A bound on the jump $\llbracket\nabla_h^sv_h\rrbracket$ is derived using the weak continuity and the trace inequality, with which we show the discrete Poincar\'e inequality for the global virtual element space.
The optimal error estimate of the canonical interpolation $I_hu$ is achieved
after establishing a Galerkin orthogonality of $u - I_hu$.
By employing the bubble function technique which was frequently used in proving the efficiency of the a posteriori error estimators, the inverse inequality for polynomials, the generalized Green's identity and the trace inequality, we acquire the norm equivalence of the standard stabilization using $l^2$ inner products of degree of freedoms on $\ker(\Pi^K)$.
The optimal error estimates are derived for the $H^m$-nonconforming virtual element methods by further estimate the consistency error.

The shape functions of the virtual element spaces are not explicitly
known; in particular, the output of the method is a vector of degrees of freedom and not
an explicit function. In order to represent explicitly the solution, one employs
some suitable polynomial projector, which is typically piecewise defined and discontinuous over the polytopal decomposition.
However, since the degrees of freedom in the interior of each element for the virtual elements can be eliminated by the static condensation,
similarly as the hybridizable discontinuous Galerkin methods,
the virtual element methods possess fewer globally decoupled degrees of freedom than the usual discontinuous Galerkin methods. Furthermore, the nonconforming virtual element can be constructed in a universal way which allows unified error analysis and is employed for theoretical purposes, independently of the way one wants to represent the solution.

The rest of this paper is organized as follows. In Section 2, we present some notations and the construction of the fully $H^1$- and $H^2$-nonconforming virtual elements. The general fully $H^m$-nonconforming virtual element is designed in Section 3. The corresponding $H^m$-nonconforming virtual element method and its error estimate are shown in Section 4 and Section 5 respectively. A conclusion is given in Section 6. In the Appendix, we prove the norm equivalence and give a remark on the implementation.

\section{Preliminaries}

\subsection{Notation}
Assume that $\Omega\subset \mathbb{R}^n~(n\geq 2)$ is a bounded
polytope.
For any nonnegative integer $r$ and $1\leq\ell\leq n$, denote the set of $r$-tensor spaces over $\mathbb R^{\ell}$ by $\mathbb T_{\ell}(r):=(\mathbb R^{\ell})^r=\prod_{j=1}^{r}\mathbb R^{\ell}$.
Given a bounded domain $K\subset\mathbb{R}^{n}$ and a
non-negative integer $k$, let $H^k(K; \mathbb T_{\ell}(r))$ be the usual Sobolev space of functions
over $K$ taking values in the tensor space $\mathbb T_{\ell}(r)$. The corresponding norm and semi-norm are denoted respectively by
$\Vert\cdot\Vert_{k,K}$ and $|\cdot|_{k,K}$. It is customary to rewrite $H^k(K; \mathbb T_{\ell}(0))$ as $H^k(K)$.
For any $F\subset\partial K$,
denote by $\nu_{K, F}$ the
unit outward normal to $\partial K$. Without causing any confusion, we will abbreviate $\nu_{K, F}$ as $\nu$ for simplicity.
Define $H_0^k(K)$ as the closure of $C_{0}^{\infty}(K)$ with
respect to the norm $\Vert\cdot\Vert_{k,K}$, i.e. (cf. \cite[Theorem~5.37]{AdamsFournier2003}),
\[
H_0^k(K):=\left\{v\in H^k(K): v=\frac{\partial v}{\partial \nu}=\cdots=\frac{\partial^{k-1} v}{\partial \nu^{k-1}}=0 \quad\text{on
}\partial K\right\},
\]
and define $H_0^1(K; \mathbb T_{\ell}(r))$ in a similar way.
Let $(\cdot, \cdot)_K$ be the standard inner product on $L^2(K; \mathbb T_{\ell}(r))$. If $K$ is $\Omega$, we abbreviate
$\Vert\cdot\Vert_{k, K}$, $|\cdot|_{k, K}$ and $(\cdot, \cdot)_K$ by $\Vert\cdot\Vert_{k}$, $|\cdot|_{k}$ and $(\cdot, \cdot)$,
respectively.
Notation $\mathbb P_k(K)$ stands for the set of all
polynomials over $K$ with the total degree no more than $k$.
And denote by $\mathbb P_{k}(K; \mathbb T_{\ell}(r))$ the tensorial version space of $\mathbb P_{k}(K)$.
Let $Q_k^{K}$ be the $L^2$-orthogonal projection onto $\mathbb P_k(K; \mathbb T_{\ell}(r))$.

For an $n$-dimensional multi-index $\alpha = (\alpha_1, \cdots , \alpha_n)$ with $\alpha_i\in\mathbb Z^+\cup\{0\}$, define $|\alpha|:=\sum_{i=1}^n\alpha_i$.
For $0\leq \ell \leq n$, let $A_{\ell}$ be the set consisting of all multi-indexes $\alpha$ with $\sum_{i=\ell+1}^n\alpha_i=0$, i.e., non-zero index only exists for $1\leq i\leq l$.
For any non-negative integer $k$, define the scaled monomial $\mathbb M_k(K)$ on an $\ell$-dimensional domain $K$
\[
\mathbb M_k(K):=\left\{\Big (\frac{\bs x-\bs x_K}{h_K}\Big )^{\alpha}, \alpha\in A_{\ell}, |\alpha|\leq k\right\},
\]
where $h_K$ is the diameter of $K$ and $\bs x_K$ is the centroid of $K$. And $\mathbb M_k(K):=\emptyset$ whenever $k<0$.

Let $\{\mathcal {T}_h\}$ be a regular family of partitions
of $\Omega$ into nonoverlapping simple polytopal elements with $h:=\max_{K\in \mathcal {T}_h}h_K$.
Let $\mathcal{F}_h^r$ be the set of all $(n-r)$-dimensional faces
of the partition $\mathcal {T}_h$ for $r=1, 2, \cdots, n$, and its boundary part
\[
\mathcal{F}_h^{r, \partial}:=\{F\in\mathcal{F}_h^r: F\subset\partial\Omega\},
\]
and interior part $\mathcal{F}_h^{r, i}:=\mathcal{F}_h^{r}\backslash \mathcal{F}_h^{r, \partial}$.
Moreover, we set for each $K\in\mathcal{T}_h$
\[
\mathcal{F}^r(K):=\{F\in\mathcal{F}_h^r: F\subset\partial K\}.
\]
The supscript $r$ in $\mathcal{F}_h^r$ represents the co-dimension of an $(n-r)$-dimensional face $F$ as we shall show later the degree of freedom will be associated to the $r$-normal vectors of $F$. Similarly, for $F\in\mathcal{F}_h^r$ and $j=0, 1, \cdots, n-r$ with $r=1, 2, \cdots, n$, we define
\[
\mathcal F^j(F):=\{e\in\mathcal{F}_h^{r+j}: e\subset\overline{F}\}.
\]
Here $j$ is the co-dimension relative to the face $F$. Apparently $\mathcal F^0(F)=F$.

For any $F\in\mathcal{F}_h^r$, let $\nu_{F,1}, \cdots, \nu_{F,r}$ be its
mutually perpendicular unit normal vectors, and define the surface gradient on $F$ as
\begin{equation}\label{eq:surfacegrad}
\nabla_{F}v:=\nabla v-\sum_{i=1}^r\frac{\partial v}{\partial\nu_{F,i}}\nu_{F,i},
\end{equation}
namely the projection of $\nabla v$ to the face $F$, which is independent of the choice of the normal vectors.
When $v$ is a tensor, the surface gradient $\nabla_{F}v$ is defined element-wisely in convention, which is a one-order higher tensor.
And denote by $\div_{F}$ the corresponding surface divergence.
For any $F\in\mathcal F^{r}_h$ and $\alpha\in A_{r}$ for $r=1,\cdots, m$, set
\[
\frac{\partial^{|\alpha|}v}{\partial\nu_{F}^{\alpha}}:=\frac{\partial^{|\alpha|}v}{\partial\nu_{F, 1}^{\alpha_1}\cdots\partial\nu_{F, r}^{\alpha_{r}}}.
\]

For any $(n-2)$-dimensional face $e\in\mathcal{F}_h^{2}$,  denote
\[
\partial^{-1}e :=\{F\in\mathcal{F}_h^{1}: e\subset \partial F\}.
\]
Similarly for any $(n-1)$-dimensional face $F\in\mathcal{F}_h^{1}$,  let
\[
\partial^{-1}F :=\{K\in\mathcal{T}_h: F\in \mathcal{F}^1(K)\}.
\]
For non-negative integers $m$ and $k$, let
\[
H^m(\mathcal T_h):=\{v\in L^2(\Omega): v|_K\in H^m(K)\textrm{ for each } K\in\mathcal T_h\},
\]
\[
\mathbb P_k(\mathcal T_h):=\{v\in L^2(\Omega): v|_K\in \mathbb P_k(K)\textrm{ for each } K\in\mathcal T_h\}.
\]
For a function $v\in H^m(\mathcal T_h)$, equip the usual broken $H^m$-type norm and semi-norm
\[
\|v\|_{m,h}:=\Big (\sum_{K\in\mathcal T_h}\|v\|_{m,K}^2\Big )^{1/2},\quad |v|_{m,h}:=\Big (\sum_{K\in\mathcal T_h}|v|_{m,K}^2\Big )^{1/2}.
\]

We introduce jumps on ($n-1$)-dimensional faces.
Consider two adjacent elements $K^+$ and $K^-$ sharing an interior ($n-1$)-dimensional face $F$.
Denote by $\nu^+$ and $\nu^-$ the unit outward normals
to the common face $F$ of the elements $K^+$ and $K^-$, respectively.
For a scalar-valued or tensor-valued function $v$, write $v^+:=v|_{K^+}$ and $v^-
:=v|_{K^-}$.   Then define the jump on $F$ as
follows:
\[
\llbracket v\rrbracket:=v^+\nu_{F,1}\cdot \nu^++v^-\nu_{F,1}\cdot \nu^-.
\]
On a face $F$ lying on the boundary $\partial\Omega$, the above term is
defined by $\llbracket v\rrbracket
   :=v\nu_{F,1}\cdot \nu.$

Throughout this paper, we also use
``$\lesssim\cdots $" to mean that ``$\leq C\cdots$", where
$C$ is a generic positive constant independent of mesh size $h$, but may depend on the chunkiness parameter of the polytope, the degree of polynomials $k$, the order of differentiation $m$ and the dimension of space $n$,
which may take different values at different appearances. And $A\eqsim B$ means $A\lesssim B$ and $B\lesssim A$.
Hereafter, we always assume $k\geq m$.

We summarize important notation in the following tables.
\begin{table}[htp]
\caption{Notation of the mesh, elements, and faces.}
\begin{center}
\begin{tabular}{|c|c|c|c|c|}
\hline
$m$ & order of differentiation $H^m$ & $n$ & dimension of space $\mathbb R^n$ & $m\leq n, k\geq m$ \\
$k$ & degree of polynomial $\mathbb P_k$ &  $r$ & co-dimension of a face & $0\leq r\leq n$\\
$\mathcal T_h$ & a mesh of $\Omega$ & $K$ & a polytope element & $K\in \mathcal T_h$\\
$\mathcal F_h^r$ & $(n-r)$-dimensional face & $F$ & a typical face & $F\in \mathcal F_h^r$\\
$\partial^{-1}e$ & all faces surrounding $e$ & $\partial^{-1}F$ & elements containing $F$ & $e\in \mathcal F_h^2, F\in \mathcal F_h^1$\\
\hline
\end{tabular}
\end{center}
\label{default1}
\end{table}%

\begin{table}[htp]
\caption{Notation for differentiation}
\begin{center}
\begin{tabular}{|c|c|}
\hline
$\alpha = (\alpha_1, \alpha_2, \ldots, \alpha_n)$ & an $n$-dimensional multi-index \\
$A_r$ set of multi-index & $\alpha = (\alpha_1, \ldots, \alpha_r, 0, \ldots, 0)$ for $\alpha \in A_r$\\
$\nu_{F,1}, \cdots, \nu_{F,r}$ & $r$ linearly independent unit normal vectors for $F\in\mathcal F_h^r$\\
$\displaystyle \nabla_{F}v:=\nabla v-\sum_{i=1}^r\frac{\partial v}{\partial\nu_{F,i}}\nu_{F,i}$ & surface gradient on $F$\\
$D^j_{F,\alpha}(v)$ & a $j$-th order derivative of $v$ on $F$\\
$\displaystyle \frac{\partial^{|\alpha|}v}{\partial\nu_{F}^{\alpha}}:=\frac{\partial^{|\alpha|}v}{\partial\nu_{F, 1}^{\alpha_1}\cdots\partial\nu_{F, r}^{\alpha_{r}}}.
$ & a multi-indexed normal derivative on $F$\\
\hline
\end{tabular}
\end{center}
\label{default2}
\end{table}%

\subsection{$H^1$-nonconforming virtual element}\label{sec:H1nvem}
To drive the $H^m$-nonconforming virtual element in a unified framework, we first revisit the simplest case for the purpose of discovering the underlying mechanism.
Taking any $K\in\mathcal T_h$, let $u\in H^2(K)$ and $v\in H^1(K)$.
Applying the integration by parts, it holds
\begin{equation}\label{eq:H1Green}
(\nabla u, \nabla v)_K=-(\Delta u, v)_K +\sum_{F\in \mathcal F^{1}(K)}(\frac{\partial u}{\partial\nu_{K,F}}, v)_F.
\end{equation}
Imaging $u\in \mathbb P_k(K)$, we are inspired by the Green's identity~\eqref{eq:H1Green} to advance the following local degrees of freedom (dofs) $\mathcal N_k(K)$ of the $H^1$ nonconforming virtual element:
\begin{align}
\frac{1}{|F|}(v, q)_F & \quad\forall~q\in\mathbb M_{k-1}(F) \textrm{ on each }  F\in\mathcal F^{1}(K), \label{H1dof1}\\
\frac{1}{|K|}(v, q)_K & \quad\forall~q\in\mathbb M_{k-2}(K). \label{H1dof2}
\end{align}
The local space of the $H^1$-nonconforming virtual element is
\[
V_k(K):=\left\{u\in H^1(K): \Delta u\in \mathbb P_{k-2}(K), \;\frac{\partial u}{\partial\nu_{K,F}}|_F\in\mathbb P_{k-1}(F) \quad\forall~F\in\mathcal F^{1}(K)\right\}
\]
for $k\geq 1$.
This is the $H^1$-nonconforming  virtual element constructed in~\cite{AyusodeDiosLipnikovManzini2016}; see also~\cite{Liu;Li;Chen:2017nonconforming}.

\subsection{$H^2$-nonconforming virtual element}\label{sec:H2nvem}
Then we consider the case $m=2$.
For each $F\in\mathcal F^1(K)$ and any function $v\in H^4(K)$, set
\begin{align*}
M_{\nu\nu}(v) &:=\nu_{F,1}^{\intercal}(\nabla^2v)\nu_{K,F}, \\
M_{\nu t}(v) &:=(\nabla^2v)\nu_{K,F} - M_{\nu\nu}(v)\nu_{F,1},\\
Q_{\nu}(v) &:=\nu_{K,F}^{\intercal}\div(\nabla^2v) + \div_{F}M_{\nu t}(v).
\end{align*}
In two dimensions, $M_{\nu\nu}(v)$, $M_{\nu t}(v)$ and $Q_{\nu}(v)$ are called normal bending moment, twisting moment and effective transverse shear force respectively when $v$ is the deflection of a thin plate in the context of elastic mechanics~\cite{FengShi1996,Reddy2006}.

\begin{lemma}
For any $u\in H^4(K)$ and $v\in H^2(K)$, it holds
\begin{align}
(\nabla^2u, \nabla^2v)_K &= (\Delta^2u, v)_K + \sum_{F\in\mathcal F^{1}(K)}\Big [(M_{\nu\nu}(u), \frac{\partial v}{\partial\nu_{F,1}})_{F} - (Q_{\nu}(u), v)_{F}\Big ] \notag \\
&\qquad+\sum_{e\in\mathcal F^{2}(K)}\sum_{F\in\mathcal F^{1}(K)\cap \partial^{-1}e}(\nu_{F,e}^{\intercal}M_{\nu t}(u), v)_{e}. \label{eq:H2Green}
\end{align}
\end{lemma}
\begin{proof}
Using integration by parts,  we get
\[
(\div(\nabla^2u), \nabla v)_K=-(\Delta^2u, v)_K + \sum_{F\in \mathcal F^{1}(K)}(\nu_{K,F}^{\intercal}\div(\nabla^2u), v)_{F},
\]
and for each $F\in\mathcal F^1(K)$,
\[
(M_{\nu t}(u), \nabla_{F}v)_{F}=-(\div_{F}M_{\nu t}(u), v)_{F} + \sum_{e\in \mathcal F^{1}(F)}(\nu_{F,e}^{\intercal}M_{\nu t}(u), v)_{e}.
\]
Then we acquire from the last two identities and
splitting the gradient into the tangential and normal components
\begin{align*}
(\nabla^2u, \nabla^2v)_K &= -(\div(\nabla^2u), \nabla v)_K+\sum_{F\in\mathcal F^{1}(K)}((\nabla^2u)\nu_{K,F}, \nabla v)_{F} \\
 &= -(\div(\nabla^2u), \nabla v)_K + \sum_{F\in\mathcal F^{1}(K)}(M_{\nu\nu}(u), \frac{\partial v}{\partial\nu_{F,1}})_{F} \\
 &\quad\;+ \sum_{F\in\mathcal F^{1}(K)}(M_{\nu t}(u), \nabla_{F}v)_{F} \\
 &= (\Delta^2u, v)_K + \sum_{F\in\mathcal F^{1}(K)}\Big [ (M_{\nu\nu}(u), \frac{\partial v}{\partial\nu_{F,1}})_{F} - (Q_{\nu}(u), v)_{F}\Big ] \\
&\quad+\sum_{F\in\mathcal F^{1}(K)}\sum_{e\in \mathcal F^{1}(F)}(\nu_{F,e}^{\intercal}M_{\nu t}(u), v)_{e},
\end{align*}
which ends the proof.
\end{proof}

Inspired by the Green's identity~\eqref{eq:H2Green}, for any element $K\in\mathcal T_h$ and integer $k\geq 2$,
the local degrees of freedom $\mathcal N_k(K)$ of the $H^2$ nonconforming virtual element are given as follows:
\begin{align}
\frac{1}{|K|}(v, q)_K & \quad\forall~q\in\mathbb M_{k-4}(K), \label{H2dof4}\\
\frac{1}{|F|}(v, q)_F & \quad\forall~q\in\mathbb M_{k-3}(F) \textrm{ on each }  F\in\mathcal F^{1}(K), \label{H2dof2}\\
\frac{1}{|F|^{(n-2)/(n-1)}}(\frac{\partial v}{\partial\nu_{F,1}}, q)_F & \quad\forall~q\in\mathbb M_{k-2}(F) \textrm{ on each }  F\in\mathcal F^{1}(K), \label{H2dof3}\\
\frac{1}{|e|}(v, q)_e & \quad\forall~q\in\mathbb M_{k-2}(e) \textrm{ on each } e\in\mathcal F^{2}(K). \label{H2dof1}
\end{align}
The local space of the $H^2$ nonconforming virtual element is
\begin{align*}
V_k(K)&:=\{ u \in H^2(K): \Delta^2u \in \mathbb P_{k-4}(K), M_{\nu\nu}(u)|_F\in\mathbb P_{k-2}(F), Q_{\nu}(u)|_F\in\mathbb P_{k-3}(F), \\
&\qquad\qquad\;\; \sum_{F\in\mathcal F^{1}(K)\cap \partial^{-1}e}\nu_{F,e}^{\intercal}M_{\nu t}(u)|_e\in\mathbb P_{k-2}(e)\quad\forall~F\in\mathcal F^{1}(K), e\in\mathcal F^{2}(K)\}.
\end{align*}

\begin{remark}\rm
In two dimensions, the degrees of freedom~\eqref{H2dof1} will be reduced to the function values on the vertices of $K$. Then the virtual element $(K, \mathcal N_k(K), V_k(K))$ is the same as that in~\cite{AntoniettiManziniVerani2018,ZhaoZhangChenMao2018}.
\end{remark}

\begin{remark}\label{rmk:mwx}\rm
If the element $K\in \mathcal T_h$ is a simplex and $k=2$,
the degrees of freedom~\eqref{H2dof4}-\eqref{H2dof2} disappear, and the degrees of freedom~\eqref{H2dof3}-\eqref{H2dof1} are the same as the Morley-Wang-Xu element's degrees of freedom in~\cite{WangXu2006}. Indeed the virtual element $(K, \mathcal N_k(K), V_k(K))$ coincides with the Morley-Wang-Xu element in~\cite{WangXu2006} when $k=2$ and $K$ is a simplex.
\end{remark}

\section{$H^m$-nonconforming Virtual Element with $1\leq m\leq n$}

In this section, we will construct the $H^m$-nonconforming virtual element. It has been illustrated in Sections~\ref{sec:H1nvem}-\ref{sec:H2nvem} that the Green's identity plays a vital role in deriving the $H^1$ and $H^2$ nonconforming virtual elements. To this end, we shall derive a generalized Green's identity for the $H^m$ space first.
\subsection{Generalized Green's identity}
For any scalar or tensor-valued smooth function $v$, nonnegative integer $j$, $F\in \mathcal{F}_h^r$ with $1\leq r\leq n$, and $\alpha\in A_{r}$, we use $D^j_{F, \alpha}(v)$ to denote some $j$-th order derivative of $v$ restrict on $F$, which may take different expressions at different appearances.

\begin{lemma}\label{lem:20190605-1}
Let $K\in\mathcal T_h$, $F\in \mathcal F^r(K)$ with $1\leq r\leq n-1$, and $s$ be a positive integer satisfying $s\leq n-r$.
There exist differential operators $D^{s-j-|\alpha|}_{e, \alpha}$ for $j=0,\cdots, s$, $e\in \mathcal F^j(F)$ and $\alpha\in A_{r+j}$ with $|\alpha|\leq s-j$ such that for any $\tau\in H^{s}(F; \mathbb T_{n}(s))$ and $(\nabla^sv)|_F\in L^2(F; \mathbb T_{n}(s))$, it holds
\begin{equation}\label{eq:HmfaceGreen}
(\tau, \nabla^sv)_F = \sum_{j=0}^s\sum_{e\in\mathcal F^j(F)}\sum_{\alpha\in A_{r+j}\atop|\alpha|\leq s-j}\Big ( D^{s-j-|\alpha|}_{e, \alpha}(\tau), \;\frac{\partial^{|\alpha|}v}{\partial\nu_{e}^{\alpha}}\Big )_e.
\end{equation}
\end{lemma}
\begin{proof}
We adopt the mathematical induction to prove the identity~\eqref{eq:HmfaceGreen}.
When $s=1$, we get from~\eqref{eq:surfacegrad} and integration by parts
\begin{align*}
(\tau, \nabla v)_F&=\sum_{i=1}^r(\tau, \frac{\partial v}{\partial\nu_{F,i}}\nu_{F,i})_F+(\tau, \nabla_Fv)_F \\
&=\sum_{i=1}^r(\nu_{F,i}^{\intercal}\tau, \frac{\partial v}{\partial\nu_{F,i}})_F - (\div_F\tau, v)_F +\sum_{e\in\mathcal F^1(F)}(\nu_{F,e}^{\intercal}\tau, v)_e.
\end{align*}
Thus the identity~\eqref{eq:HmfaceGreen} holds for $s=1$.

Next assume the identity~\eqref{eq:HmfaceGreen} is true for $s=\ell-1$ with $2\leq\ell\leq n-r$, then let us prove it is also true for $s=\ell$.
We get from~\eqref{eq:surfacegrad} and integration by parts
\begin{align*}
(\tau, \nabla^\ell v)_F&=\sum_{i=1}^r(\tau\nu_{F,i}, \nabla^{\ell-1}\frac{\partial v}{\partial\nu_{F,i}})_F+(\tau, \nabla_F\nabla^{\ell-1}v)_F \\
&=\sum_{i=1}^r(\tau\nu_{F,i}, \nabla^{\ell-1}\frac{\partial v}{\partial\nu_{F,i}})_F - (\div_F\tau, \nabla^{\ell-1}v)_F +\sum_{e\in\mathcal F^1(F)}(\tau\nu_{F,e}, \nabla^{\ell-1}v)_e.
\end{align*}
Applying the assumption with $s=\ell-1$ to the right hand side of the last equation term by term, we have
\[
(\tau\nu_{F,i}, \nabla^{\ell-1}\frac{\partial v}{\partial\nu_{F,i}})_F
=\sum_{j=0}^{\ell-1}\sum_{e\in\mathcal F^j(F)}\sum_{\alpha\in A_{r+j}\atop|\alpha|\leq {\ell-1}-j}\Big (D^{\ell-1-j-|\alpha|}_{e, \alpha}(\tau\nu_{F,i}), \;\frac{\partial^{|\alpha|}}{\partial\nu_{e}^{\alpha}}\Big (\frac{\partial v}{\partial\nu_{F,i}}\Big )\Big )_e,
\]
\[
 (\div_F\tau, \nabla^{\ell-1}v)_F=\sum_{j=0}^{\ell-1}\sum_{e\in\mathcal F^j(F)}\sum_{\alpha\in A_{r+j}\atop|\alpha|\leq {\ell-1}-j}\Big (D^{\ell-1-j-|\alpha|}_{e, \alpha}(\div_F\tau), \;\frac{\partial^{|\alpha|}v}{\partial\nu_{e}^{\alpha}}\Big )_e,
\]
\[
(\tau\nu_{F,e}, \nabla^{\ell-1}v)_e=\sum_{j=0}^{\ell-1}\sum_{\tilde e\in\mathcal F^j(e)}\sum_{\alpha\in A_{r+1+j}\atop|\alpha|\leq {\ell-1}-j}\Big (D^{\ell-1-j-|\alpha|}_{\tilde e, \alpha}(\tau\nu_{F,e}), \;\frac{\partial^{|\alpha|}v}{\partial\nu_{\tilde e}^{\alpha}}\Big )_{\tilde e}.
\]
Finally we conclude~\eqref{eq:HmfaceGreen} for $s=\ell$ by combining the last fourth equations and the fact that $\nu_{F,i}$ is a linear combination of $\nu_{e, 1},\cdots, \nu_{e, r+j}$ if $e\in\mathcal F^j(F)$.
\end{proof}

For each term in the right hand side of~\eqref{eq:HmfaceGreen}, the total number of differentiation of the integrand is $s-j$. In view of Stokes theorem, $e\in \mathcal F^j(F)$ can be thought as $e\in \partial^j(F)$ so that the total number of differentiation is $s$ which matches that of the left hand side.
The bounds on $s$ in Lemma~\ref{lem:20190605-1} imply $r+s\leq n$, then $\mathcal F^{s}(F)\subset\mathcal F^{r+s}(K)$ is well-defined for $F\in \mathcal F^r(K)$. Hence we can recursively apply the Stokes theorem till there is no derivative of $v$ on the lowest dimensional faces.

We give two examples of identity \eqref{eq:HmfaceGreen}.
When $n\geq 2$, $s=1$, and $1\leq r\leq n-1$, the explicit expression of \eqref{eq:HmfaceGreen} is that for any $F\in \mathcal F^r(K)$, $\tau\in H^{1}(F; \mathbb R^{n})$ and $(\nabla v)|_F\in L^2(F; \mathbb R^{n})$,
\[
(\tau, \nabla v)_F= - (\div_F\tau, v)_F+ \sum_{i=1}^r(\nu_{F,i}^{\intercal}\tau, \frac{\partial v}{\partial\nu_{F,i}})_F +\sum_{e\in\mathcal F^1(F)}(\nu_{F,e}^{\intercal}\tau, v)_e.
\]
If $n=3$ and $s=2$, then $r=1$. And the explicit expression of \eqref{eq:HmfaceGreen} is that for any $F\in \mathcal F^1(K)$, $\tau\in H^{2}(F; \mathbb T_{3}(2))$ and $(\nabla^2v)|_F\in L^2(F; \mathbb T_{3}(2))$,
\begin{align*}
(\tau, \nabla^2 v)_F=&\,(\div_F\div_F\tau, v)_F - (\nu_{F,1}^{\intercal}(\div_F\tau)+\div_F(\tau\nu_{F,1}), \frac{\partial v}{\partial\nu_{F,1}})_F \\
& + (\nu_{F,1}^{\intercal}\tau\nu_{F,1}, \frac{\partial^2 v}{\partial\nu_{F,1}^2})_F - \sum_{e\in\mathcal F^1(F)}(\nu_{F,e}^{\intercal}(\div_F\tau)+\div_e(\tau\nu_{F,e}), v)_e \\
&+\sum_{e\in\mathcal F^1(F)}\sum_{i=1}^2(\nu_{e,i}^{\intercal}\tau\nu_{F,e} + (\nu_{e,i}^{\intercal}\nu_{F,1})\nu_{F,1}^{\intercal}\tau\nu_{F,e}, \frac{\partial v}{\partial\nu_{e,i}})_e \\
&+\sum_{e\in\mathcal F^1(F)}\sum_{\delta\in\mathcal F^1(e)}(\nu_{e,\delta}^{\intercal}\tau\nu_{F,e})(\delta)v(\delta).
\end{align*}

\begin{theorem}\label{thm:HmGreen}
Let $1\leq m\leq n$.
There exist differential operators $D^{2m-j-|\alpha|}_{F, \alpha}$ for $j=1,\cdots, m$, $F\in\mathcal F^j(K)$ and $\alpha\in A_j$ with $|\alpha|\leq m-j$ such that it holds for any $u\in H^{m}(K)$ satisfying $(-\Delta)^mu\in L^2(K)$ and $D^{2m-j-|\alpha|}_{F, \alpha}u\in L^2(F)$,  and any $v\in H^m(K)$
\begin{align}
(\nabla^mu, \nabla^mv)_K = &((-\Delta)^mu, v)_K \notag\\
&+ \sum_{j=1}^m\sum_{F\in\mathcal F^j(K)}\sum_{\alpha\in A_{j}\atop|\alpha|\leq m-j}\Big (D^{2m-j-|\alpha|}_{F, \alpha}u, \frac{\partial^{|\alpha|}v}{\partial\nu_{F}^{\alpha}}\Big )_F. \label{eq:HmGreen}
\end{align}
\end{theorem}
\begin{proof}
By the density argument, we can assume $u\in H^{2m}(K)$.
We still use the mathematical induction to prove the identity~\eqref{eq:HmGreen}.
The identity~\eqref{eq:HmGreen} for $m=1$ and $m=2$ is just the identities~\eqref{eq:H1Green} and~\eqref{eq:H2Green} respectively.
Assume the identity~\eqref{eq:HmGreen} is true for $m=\ell-1$ with $3\leq\ell\leq n$, then let us prove the identity~\eqref{eq:HmGreen} is also true for $m=\ell$.

Applying the integration by parts,
\begin{align*}
(\nabla^{\ell}u, \nabla^{\ell}v)_K&=-(\div\nabla^{\ell}u, \nabla^{\ell-1}v)_K +\sum_{F\in\mathcal F^{1}(K)}((\nabla^{\ell}u)\nu_{K,F}, \nabla^{\ell-1} v)_{F} \\
&=(\nabla^{\ell-1}(-\Delta u), \nabla^{\ell-1}v)_K +\sum_{F\in\mathcal F^{1}(K)}((\nabla^{\ell}u)\nu_{K,F}, \nabla^{\ell-1} v)_{F}.
\end{align*}
Since the identity~\eqref{eq:HmGreen} holds for $m=\ell-1$, we have
\begin{align*}
&(\nabla^{\ell-1}(-\Delta u), \nabla^{\ell-1}v)_K \\
= &((-\Delta)^{\ell}u, v)_K + \sum_{j=1}^{\ell-1}\sum_{F\in\mathcal F^j(K)}\sum_{\alpha\in A_{j}\atop|\alpha|\leq {\ell-1}-j}\Big (D^{2({\ell-1})-j-|\alpha|}_{F, \alpha}(-\Delta u), \frac{\partial^{|\alpha|}v}{\partial\nu_{F}^{\alpha}}\Big )_F.
\end{align*}
Taking $\tau=(\nabla^{\ell}u)\nu_{K,F}$, $s=\ell-1$ and $r=1$ in~\eqref{eq:HmfaceGreen}, we get
\begin{align*}
((\nabla^{\ell}u)\nu_{K,F}, \nabla^{\ell-1}v)_F
= &\sum_{j=0}^{\ell-1}\sum_{e\in\mathcal F^j(F)}\sum_{\alpha\in A_{1+j}\atop|\alpha|\leq {\ell-1}-j}\Big (D^{{\ell-1}-j-|\alpha|}_{e, \alpha}((\nabla^{\ell}u)\nu_{K,F}), \;\frac{\partial^{|\alpha|}v}{\partial\nu_{e}^{\alpha}}\Big )_e\\
= &\sum_{j=1}^{\ell}\sum_{e\in\mathcal F^{j-1}(F)}\sum_{\alpha\in A_{j}\atop|\alpha|\leq {\ell}-j}\Big (D^{{\ell}-j-|\alpha|}_{e, \alpha}((\nabla^{\ell}u)\nu_{K,F}), \;\frac{\partial^{|\alpha|}v}{\partial\nu_{e}^{\alpha}}\Big )_e.
\end{align*}
Therefore we finish the proof by combining the last three equations.
\end{proof}

Examples for $m=1, 2, n\geq m$ and $m = n = 3$ can be found in Appendix~\ref{sec:appendixB}.

\subsection{Virtual element space}
Inspired by identity~\eqref{eq:HmGreen}, for any element $K\in\mathcal T_h$ and integer $k\geq m$,
the local degrees of freedom $\mathcal N_k(K)$ are given as follows:
\begin{align}
\frac{1}{|K|}(v, q)_K & \quad\forall~q\in\mathbb M_{k-2m}(K), \label{Hmdof1}\\
\frac{1}{|F|^{(n-j-|\alpha|)/(n-j)}}(\frac{\partial^{|\alpha|}v}{\partial\nu_{F}^{\alpha}}, q)_F & \quad\forall~q\in\mathbb M_{k-(2m-j-|\alpha|)}(F) \label{Hmdof2}
\end{align}
on each $F\in\mathcal F^{j}(K)$, where $j=1,\cdots, m$, $\alpha\in A_{j}$ and $|\alpha|\leq m-j$.

We present an heuristic explanation of the scaling factor in (3.4).
Let $\hat K=\{\hat{\bs x}\in \mathbb R^n: \hat{\bs x}=\frac{1}{h_K}(\bs x-\bs x_K)\quad \forall~\bs x\in K\}$, and an affine mapping $\Psi: \hat{\bs x}\in\mathbb R^n\to \Psi(\hat{\bs x})=h_K\hat{\bs x}+\bs x_K\in\mathbb R^n$.
Then $h_{\hat K}\eqsim 1$ and $\Psi(\hat K)=K$.
For any function $v(\bs x)$ defined on $K$, let $\hat v(\hat{\bs x}):= v(\Psi(\hat{\bs x}))$, which is defined on $\hat K$.
By the scaling argument, we have
\[
(\frac{\partial^{|\alpha|}v}{\partial\nu_{F}^{\alpha}}, q)_F = h_K^{n-j-|\alpha|}(\frac{\partial^{|\alpha|}\hat v}{\partial\nu_{\hat F}^{\alpha}}, \hat q)_{\hat F}.
\]
By the mesh conditions (A1)-(A2) in Section~\ref{subsec:meshcontitions}, it holds $|F|\eqsim h_K^{n-j}$, thus there exists a constant $C>0$ being independent of $h_K$ such that $|F|^{(n-j-|\alpha|)/(n-j)}=C h_K^{n-j-|\alpha|}$.
Then
\[
\frac{1}{|F|^{(n-j-|\alpha|)/(n-j)}}(\frac{\partial^{|\alpha|}v}{\partial\nu_{F}^{\alpha}}, q)_F=\frac{1}{C}(\frac{\partial^{|\alpha|}\hat v}{\partial\nu_{\hat F}^{\alpha}}, \hat q)_{\hat F}=\frac{1}{C_1}\frac{1}{|\hat F|^{(n-j-|\alpha|)/(n-j)}}(\frac{\partial^{|\alpha|}\hat v}{\partial\nu_{\hat F}^{\alpha}}, \hat q)_{\hat F},
\]
where $C_1=C/|\hat F|^{(n-j-|\alpha|)/(n-j)}$ is independent of $h_K$. Hence all the degrees of freedom in \eqref{Hmdof1}-\eqref{Hmdof2} share the same order of magnitude.

Again due to the first terms in the inner products of the right hand side of~\eqref{eq:HmGreen}, and the degrees of freedom~\eqref{Hmdof1}-\eqref{Hmdof2}, it is inherent to define the local space of the $H^m$-nonconforming virtual element as
\begin{align*}
V_k(K):=\{ u\in H^m(K):& (-\Delta)^m u\in \mathbb P_{k-2m}(K),   \\
& D^{2m-j-|\alpha|}_{F, \alpha}(u)|_F\in\mathbb P_{k-(2m-j-|\alpha|)}(F)\quad\forall~ F\in\mathcal F^{j}(K),\\
 &\qquad\qquad\qquad j=1,\cdots, m, \,\alpha\in A_{j} \textrm{ and } |\alpha|\leq m-j\},
\end{align*}
where the differential operators $D^{2m-j-|\alpha|}_{F, \alpha}$ are introduced in Theorem~\ref{thm:HmGreen}.

In the following we shall prove that $(K, \mathcal N_k(K), V_k(K))$ forms a finite element triple in the sense of Ciarlet~\cite{Ciarlet1978}. Unlike the traditional finite element, in virtual element, only the set of the degrees of freedom $\mathcal N_k(K)$ needs to be explicitly known. The `virtual' space $V_k(K)$ is only needed for the purpose of analysis and the specific formulation for $D^{2m-j-|\alpha|}_{F, \alpha}$ is not needed in the definition of $V_k(K)$.

The following property is the direct result of~\eqref{eq:HmfaceGreen} and the definition of the degrees of freedom~\eqref{Hmdof2}.
\begin{lemma}\label{lem:bimgunisolvence}
Let $K\in\mathcal T_h$, $F\in \mathcal F^r(K)$ with $1\leq r\leq m$, nonnegative integer $s\leq m-r$ satisfying $k\geq 2m-(r+s)$.
For any $\tau\in \mathbb P_{k-(2m-r-s)}(F; \mathbb T_{n}(s))$ and $(\nabla^sv)|_F\in L^2(F; \mathbb T_{n}(s))$, the term
\begin{equation*}
(\tau, \nabla^sv)_F
\end{equation*}
is uniquely determined by the degrees of freedom $\Big (\frac{\partial^{|\alpha|}v}{\partial\nu_{e}^{\alpha}}, q\Big )_e$ for all $0\leq j\leq s$, $e\in\mathcal F^j(F)$, $\alpha\in A_{r+j}$ with $|\alpha|\leq s-j$, and $q\in\mathbb M_{k-(2m-r-j-|\alpha|)}(e)$.
\end{lemma}

\begin{lemma}
We have $\mathbb P_{k}(K)\subseteq V_k(K)$ and
\begin{equation}\label{eq:dim}
\dim V_k(K)= \dim \mathcal N_k(K).
\end{equation}
\end{lemma}
\begin{proof}
For any $q\in\mathbb P_{k}(K)$, it is obvious that
\[
(-\Delta)^m q\in \mathbb P_{k-2m}(K), \quad D^{2m-j-|\alpha|}_{F, \alpha}q|_F\in\mathbb P_{k-(2m-j-|\alpha|)}(F).
\]
Hence it holds $\mathbb P_{k}(K)\subseteq V_k(K)$. Since all the differential operators in the definition of $V_k(K)$ are linear, then $V_k(K)$ is a vector space.

Next we count the dimension of $V_k(K)$. Consider the local polyharmonic equation with the Neumann boundary condition
\begin{equation}\label{eq:localmharmonic}
\left\{
\begin{aligned}
&(-\Delta)^mu=f_1\quad\text{in
}K, \\
&D^{2m-j-|\alpha|}_{F,\alpha}(u)=g_{j}^{F,\alpha}\quad\text{on each }F\in\mathcal F^{j}(K) \\
&\qquad\qquad\qquad\qquad\qquad \text{with } j=1,\cdots, m, \,\alpha\in A_{j} \textrm{ and } |\alpha|\leq m-j,
\end{aligned}
\right.
\end{equation}
where $f_1\in\mathbb P_{k-2m}(K)$, $g_{j}^{F,\alpha}\in\mathbb P_{k-(2m-j-|\alpha|)}(F)$.
Applying the generalized Green's identity \eqref{eq:HmGreen}, the weak formulation of \eqref{eq:localmharmonic} is:
\begin{equation}\label{eq:20190902}
(\nabla^mu, \nabla^mv)_K = (f_1, v)_K + \sum_{j=1}^m\sum_{F\in\mathcal F^j(K)}\sum_{\alpha\in A_{j}\atop|\alpha|\leq m-j}\Big(g_{j}^{F,\alpha}, \frac{\partial^{|\alpha|}v}{\partial\nu_{F}^{\alpha}}\Big )_F
\end{equation}
for any $v\in H^m(K)$.
If taking $v=q\in\mathbb P_{m-1}(K)$ in \eqref{eq:20190902}, we have the compatibility condition of the data
\begin{equation}\label{eq:compatibility}
(f_1, q)_K + \sum_{j=1}^m\sum_{F\in\mathcal F^j(K)}\sum_{\alpha\in A_{j}\atop|\alpha|\leq m-j}\Big(g_{j}^{F,\alpha}, \frac{\partial^{|\alpha|}q}{\partial\nu_{F}^{\alpha}}\Big )_F=0\quad\forall~q\in\mathbb P_{m-1}(K).
\end{equation}

On the other hand, given $f_1\in\mathbb P_{k-2m}(K)$, $g_{j}^{F,\alpha}\in\mathbb P_{k-(2m-j-|\alpha|)}(F)$ satisfying the compatibility condition \eqref{eq:compatibility}, the weak formulation of the Neumann problem of the local polyharmonic equation \eqref{eq:localmharmonic} is: find $u\in H^m(K)/\mathbb P_{m-1}(K)$ such that
\[
(\nabla^mu, \nabla^mv)_K = (f_1, v)_K + \sum_{j=1}^m\sum_{F\in\mathcal F^j(K)}\sum_{\alpha\in A_{j}\atop|\alpha|\leq m-j}\Big (g_{j}^{F,\alpha}, \frac{\partial^{|\alpha|}v}{\partial\nu_{F}^{\alpha}}\Big )_F
\]
for all $v\in H^m(K)/\mathbb P_{m-1}(K)$.
The well-posedness of this variational formulation is guaranteed by the Lax-Milgram lemma~\cite{Babuvska1971,Ciarlet1978} and specifically the well-posedness of polyharmonic equations with various boundary conditions can be found in \cite{GazzolaGrunauSweers2010,AgmonDouglisNirenberg1959}.

Therefore $\dim(V_k(K)/\mathbb P_{m-1}(K))$ equals
\[
\dim\mathbb P_{k-2m}(K) + \sum_{j=1}^m\sum_{F\in\mathcal F^j(K)}\sum_{\alpha\in A_{j}\atop|\alpha|\leq m-j}\dim\mathbb P_{k-(2m-j-|\alpha|)}(F) -\dim\mathbb P_{m-1}(K),
\]
where the dimension of the constraint for the data is subtracted. When counting $\dim V_k(K)$, we should add back the dimension of the kernel space, i.e., solution spaces of $(\nabla^mu, \nabla^mv)_K = 0$ which implies
$\dim V_k(K) = \dim\mathbb P_{k-2m}(K) + \sum\limits_{j=1}^m\sum\limits_{F\in\mathcal F^j(K)}\sum\limits_{\alpha\in A_{j}\atop|\alpha|\leq m-j}\dim\mathbb P_{k-(2m-j-|\alpha|)}(F) = \dim \mathcal N_k(K)$.
\end{proof}

\begin{lemma}\label{eq:dofsunisolvent}
The degrees of freedom~\eqref{Hmdof1}-\eqref{Hmdof2} are unisolvent for the local virtual element space $V_k(K)$.
\end{lemma}
\begin{proof}
Let $v\in V_k(K)$ and suppose all the degrees of freedom~\eqref{Hmdof1}-\eqref{Hmdof2} vanish. We get from~\eqref{eq:20190902}
\[
\|\nabla^mv\|_{0,K}^2=0.
\]
Thus $v\in\mathbb P_{m-1}(K)$. By Lemma \ref{lem:bimgunisolvence} with $s=m-r$, we have for any $F\in \mathcal F^r(K)$ with $1\leq r\leq m$
\begin{equation}\label{eq:temp20180711}
(\tau, \nabla^sv)_F=0\quad \forall~\tau\in \mathbb P_{0}(F; \mathbb T_{n}(s)).
\end{equation}
Due to~\eqref{eq:temp20180711} with $r=1$ and the fact $v\in\mathbb P_{m-1}(K)$, it follows $v\in\mathbb P_{m-2}(K)$.
Recursively applying~\eqref{eq:temp20180711} with $r=2, \cdots, m$ gives $v=0$.
This ends the proof.
\end{proof}

\begin{remark}\label{rmk:wx}\rm
If the element $K\in \mathcal T_h$ is a simplex and $k=m$,
the degrees of freedom~\eqref{Hmdof1} disappear, and the degrees of freedom~\eqref{Hmdof2} are same as those of the nonconforming finite element in~\cite{WangXu2013}. Since $\mathbb P_m(K)\subseteq V_k(K)$, the virtual element $(K, \mathcal N_k(K), V_k(K))$ coincides with the nonconforming finite element in~\cite{WangXu2013} when $K$ is a simplex and $k=m$, which is the minimal finite element for the $2m$-th order partial differential equations in $\mathbb R^n$. In other words, we generalize the nonconforming finite element in~\cite{WangXu2013} to high order $k>m$ and arbitrary polytopes.
\end{remark}

\subsection{Local projections}
For each $K\in\mathcal T_h$, define a local $H^m$ projection $\Pi^K: H^m(K)\to\mathbb P_k(K)$ as follows: given $v\in H^m(K)$, let $\Pi_k^Kv\in\mathbb P_k(K)$ be the solution of the problem 
\begin{align}
(\nabla^m\Pi_k^Kv, \nabla^mq)_K&=(\nabla^mv, \nabla^mq)_K\quad  \forall~q\in \mathbb P_k(K),\label{eq:H2projlocal1}\\
\sum_{F\in\mathcal F^{r}(K)}Q_0^{F}(\nabla^{m-r}\Pi_k^Kv)&=\sum_{F\in\mathcal F^{r}(K)}Q_0^{F}(\nabla^{m-r}v), \quad r=1,\cdots, m.\label{eq:H2projlocal2}
\end{align}
The number of equations in~\eqref{eq:H2projlocal2} is
\[
\sum_{r=1}^mC_{n+m-1-r}^{n-1}=C_{n+m-1}^{n}=\dim(\mathbb P_{m-1}(K)).
\]
Then the well-posedness of~\eqref{eq:H2projlocal1}-\eqref{eq:H2projlocal2} can be shown by the similar argument as in the proof of Lemma \ref{eq:dofsunisolvent}. To simplify the notation, we will write as $\Pi^K$.

Obviously we have
\begin{equation}\label{eq:Hmprojbound}
|\Pi^Kv|_{m,K}\leq |v|_{m,K}\quad \forall~v\in H^m(K).
\end{equation}
We show the projection $\Pi^Ku$ is computable using the degrees of freedom~\eqref{Hmdof1}-\eqref{Hmdof2}.
\begin{lemma}\label{lm:Pi}
The operator $\Pi^K: H^m(K)\to\mathbb P_k(K)$ is a projector, i.e.
\begin{equation}\label{eq:projectPk}
\Pi^Kv=v\quad\forall~v\in\mathbb P_k(K),
\end{equation}
and the projector $\Pi^K$ can be computed using only the degrees of freedom~\eqref{Hmdof1}-\eqref{Hmdof2}.
\end{lemma}
\begin{proof}
We first show that $\Pi^K$ is a projector. Let $p=\Pi^Kv-v\in\mathbb P_k(K)$. Taking $q=p$ in~\eqref{eq:H2projlocal1}, we get $\nabla^mp=0$, i.e. $p\in\mathbb P_{m-1}(K)$.
By~\eqref{eq:H2projlocal2},
\[
\sum_{F\in\mathcal F^{r}(K)}Q_0^{F}(\nabla^{m-r}p)=0, \quad r=1,\cdots, m.
\]
Therefore $p=0$, which means
$\Pi^K$ is a projector.

Next by applying the identity~\eqref{eq:HmGreen}, the right hand side of~\eqref{eq:H2projlocal1}
\[
(\nabla^mv, \nabla^mq)_K = (v, (-\Delta)^mq)_K + \sum_{j=1}^m\sum_{F\in\mathcal F^j(K)}\sum_{\alpha\in A_{j}\atop|\alpha|\leq m-j}\Big (\frac{\partial^{|\alpha|}v}{\partial\nu_{F}^{\alpha}}, D^{2m-j-|\alpha|}_{F, \alpha}(q)\Big )_F.
\]
Hence we conclude from the degrees of freedom~\eqref{Hmdof1}-\eqref{Hmdof2} and Lemma \ref{lem:bimgunisolvence} with $s=m-r$ that the right hand sides of~\eqref{eq:H2projlocal1}-\eqref{eq:H2projlocal2} are computable.
\end{proof}
\begin{remark}\rm
 $D^{2m-j-|\alpha|}_{F, \alpha}$ is needed in the computation of $\Pi^K$. But since $q\in \mathbb P_{k}(K)$, $\nabla^m q \in \mathbb P_{k-m}(K)$ and few terms are left for moderate $k$.
\end{remark}

Let $W_k(K) :=V_k(K)$ for $k\geq3m-1$ or $m\leq k\leq 2m - 1$.
To compute the $L^2$ projection onto $\mathbb P_{m-1}(K)$ for $2m\leq k<3m-1$, following~\cite{AhmadAlsaediBrezziMariniEtAl2013}, define
\begin{align*}
\widetilde V_k(K)&:=\{v\in H^m(K): (-\Delta)^mv\in \mathbb P_{m-1}(K), D^{2m-j-|\alpha|}_{F, \alpha}(v)|_F\in\mathbb P_{k-(2m-j-|\alpha|)}(F),  \\
&\qquad\qquad\qquad\qquad\quad\;\;\,\forall~ F\in\mathcal F^{j}(K),\, j=1,\cdots, m, \,\alpha\in A_{j} \textrm{ and } |\alpha|\leq m-j\},
\end{align*}
\[
W_k(K) :=\{v\in \widetilde V_k(K): (v-\Pi^Kv, q)_K=0\quad\forall~q\in\mathbb P_{k-2m}^{\perp}(K)\},
\]
where $\mathbb P_{k-2m}^{\perp}(K)\subset\mathbb P_{m-1}(K)$ is the orthogonal complement space of $\mathbb P_{k-2m}(K)$ of $\mathbb P_{m-1}(K)$ with respect to the inner product $(\cdot,\cdot)_K$.
It is apparent that $\mathbb P_k(K)\subset W_k(K)$ and the local space $W_k(K)$ shares the same degrees of freedom as $V_k(K)$. That is for the same $\mathcal N_k(K)$, we can associate different `virtual' spaces and thus have different interpretation.

\begin{lemma}\label{eq:dofsunisolventWK}
The degrees of freedom~\eqref{Hmdof1}-\eqref{Hmdof2} are unisolvent for the local virtual element space $W_k(K)$.
\end{lemma}
\begin{proof}
It is enough to only consider the case $2m\leq k<3m-1$.
Take any $v\in W_k(K)$ all of whose degrees of freedom~\eqref{Hmdof1}-\eqref{Hmdof2} disappear. Then $\Pi^Kv=0$. By the definition of $W_k(K)$, we have
\[
(v, q)_K=0\quad\forall~q\in\mathbb P_{k-2m}^{\perp}(K),
\]
which together with~\eqref{Hmdof1} implies
\[
(v, q)_K=0\quad\forall~q\in\mathbb P_{m-1}(K).
\]
Applying the argument in Lemma~\ref{eq:dofsunisolvent} to the space $\widetilde V_k(K)$ with vanishing degrees of freedom~\eqref{Hmdof2} and the last equation, we know that $v=0$.
\end{proof}

In the original space $V_k(K)$, the volume moment, cf. \eqref{Hmdof1}, is only defined up to degree $k-2m$ which cannot compute the $L^2$-projection to $\mathbb P_{m-1}$ when $k$ is small.
For $2m\leq k<3m-1$, a desirable property of the local virtual element space $W_k(K)$ is that the $L^2$ projection $Q_{m-1}^K$ is computable if all the degrees of freedom~\eqref{Hmdof1}-\eqref{Hmdof2} are known.
Indeed it follows from the definition of $W_k(K)$
\[
(Q_{m-1}^K-Q_{k-2m}^K)(v-\Pi^{K}v)=Q_{m-1}^K(I-Q_{k-2m}^K)(v-\Pi^{K}v)=0\quad\forall~v\in W_k(K),
\]
which provides a way to compute the $L^2$ projection
\begin{equation}\label{eq:20190605}
Q_{m-1}^Kv=Q_{k-2m}^Kv+Q_{m-1}^K\Pi^{K}v-Q_{k-2m}^K\Pi^{K}v\quad\forall~v\in W_k(K).
\end{equation}


Denote by $I_K: H^m(K)\to W_k(K)$ the canonical interpolation operator based on the degrees of freedom in~\eqref{Hmdof1}-\eqref{Hmdof2}. Namely given a $u\in H^m(K)$, $I_K u\in  W_k(K)$ so that $\chi (u) = \chi (I_K u)$ for all $\chi \in \mathcal N_k(K)$. As a direct corollary of Lemma \ref{lm:Pi}, we have the following identity.
\begin{corollary}
For any $v\in H^m(K)$, it holds
\begin{equation}\label{eq:20181012-1}
\Pi^K(v)=\Pi^K(I_Kv).
\end{equation}
\end{corollary}

\section{Discrete Method}

We will present the virtual element method for the polyharmonic equation based on the virtual element $(K, \mathcal N_k(K), V_k(K))$ or $(K, \mathcal N_k(K), W_k(K))$ when $L^2$-projection is needed.

\subsection{Discretization}
Consider the polyharmonic equation with homogeneous Dirichlet boundary condition
\begin{equation}\label{eq:polyharmonic}
\begin{cases}
(-\Delta)^mu=f\qquad\qquad\qquad\quad\text{in
}\Omega, \\
u=\frac{\partial u}{\partial \nu}=\cdots=\frac{\partial^{m-1} u}{\partial \nu^{m-1}}=0 \quad\text{on
}\partial\Omega,
\end{cases}
\end{equation}
where $f\in L^2(\Omega)$ and $\Omega\subset \mathbb{R}^n$ with $1\leq m\leq n$. 
The weak formulation of the polyharmonic equation~\eqref{eq:polyharmonic} is to find $u\in H_0^m(\Omega)$ such that
\begin{equation}\label{eq:polyharmonicvar}
(\nabla^mu, \nabla^mv)=(f, v)\quad \forall~v\in H_0^m(\Omega).
\end{equation}
Since
\[
\|\nabla^mv\|_{0}\eqsim \|v\|_{m}\quad\textrm{and}\quad (f,v)\lesssim \|f\|_0\|v\|_{m}\quad\;\;\forall~v\in H_0^m(\Omega),
\]
it follows from the Lax-Milgram lemma that the variational formulation~\eqref{eq:polyharmonicvar} is well-posed.

Define the global virtual element space as
\begin{align*}
V_h:=\{v_h\in L^2(\Omega): & v_h|_K\in W_k(K)\textrm{ for each } K\in\mathcal T_h; \; \textrm{the degrees of freedom} \\
&(\frac{\partial^{|\alpha|}v_h}{\partial\nu_{F}^{\alpha}}, q)_F  \textrm{ are continuous through } F \textrm{ for all } \\
&q\in\mathbb P_{k-(2m-j-|\alpha|)}(F), \alpha\in A_{j} \textrm{ with } |\alpha|\leq m-j, F\in\mathcal F^{j}(K), \\
&\qquad\qquad\quad\textrm{ and } j=1,\cdots, m;\;(\frac{\partial^{|\alpha|}v_h}{\partial\nu_{F}^{\alpha}}, q)_F=0 \;\; \textrm{ if } F\subset\partial\Omega \}.
\end{align*}

Define the local bilinear form $a_{h,K}(\cdot, \cdot): W_k(K)\times W_k(K)\to\mathbb R$ as
\[
a_{h,K}(w, v):=(\nabla^m\Pi^Kw, \nabla^m\Pi^Kv)_K+S_K(w-\Pi^Kw, v-\Pi^Kv),
\]
where the stabilization term
\begin{equation}\label{eq:stabilization}
S_K(w, v):=h_K^{n-2m}\sum_{i=1}^{N_K}\chi_i(w)\chi_i(v)\quad\forall~ w, v\in W_k(K),
\end{equation}
where $\chi_i$ is the $i$th local degree of freedom in~\eqref{Hmdof1}-\eqref{Hmdof2} for $i=1,\cdots, N_K$.
The global bilinear form $a_h(\cdot, \cdot): V_h\times V_h\to\mathbb R$ is
\[
a_h(w_h, v_h):=\sum_{K\in\mathcal T_h}a_{h,K}(w_h, v_h).
\]

\begin{remark}\rm
The stabilization \eqref{eq:stabilization} resembles the original recipe in \cite{BeiraoBrezziCangianiManziniEtAl2013,BeiraoBrezziMariniRusso2014}.
This classical stabilization is easy to implement, and used to develop the discrete Galerkin orthogonality \eqref{eq:20181012-2} and  the stability bounds
in the Appendix~\ref{sec:appendixA}.
However, the stabilization \eqref{eq:stabilization} usually suffers from conditioning
and stability issues, especially for high-order $k$ and differential problems with
large $m$. Instead, several other stabilizations have been devised and investigated in \cite{Beirao-da-Veiga;Lovadina;Russo:2017Stability,BeiraodaVeigaDassiRusso2017,BeiraodaVeigaChernovMascottoRusso2018,Mascotto2018,DassiMascotto2018} to cure these issues.
The numerical results in \cite{Mascotto2018} show that these stabilizations have almost the same effect on the condition number of the stiffness matrix for $n=2$.
The classical stabilization, the D-recipe stabilization in \cite{BeiraodaVeigaDassiRusso2017} and the D-recipe stabilization with only boundary dofs were compared in three dimensions in \cite{DassiMascotto2018}, as a result the D-recipe stabilization outperforms the other two.


\end{remark}

Define $\Pi_h: V_h \to \mathbb P_{k}(\mathcal T_h)$ as $ (\Pi_h v) |_{K} : = \Pi^K (v|_{K})$ for each $K\in\mathcal T_h$, and
let $Q_h^l: L^2(\Omega)\to \mathbb P_{l}(\mathcal T_h)$ be the $L^2$-orthogonal projection onto $\mathbb P_{l}(\mathcal T_h)$: for any $v\in L^2(\Omega)$,
\[
(Q_h^lv)|_K:= Q_{l}^{K}(v|_K)\quad\forall~K\in\mathcal T_h.
\]
We compute the right hand side according to the following cases
\begin{equation}\label{eq:fvh}
\langle f, v_h \rangle : =
\begin{cases}
\, (f, \Pi_h v_h ), & m\leq k \leq 2m -1, \\
\, (f, Q_h^{m-1} v_h ), & 2m \leq k \leq 3m - 2,\\
\, (f, Q_h^{k-2m} v_h ), & 3m-1\leq k.
\end{cases}
\end{equation}
We will need this definition of the right hand side in order to get an optimal order of convergence, see Lemma~\ref{lem:consistencyerror}.
\begin{remark}\rm
When $m\leq k\leq 2m-1$, which is an important range for large $m$ as high order methods are harder to implement, there is no need to compute a new projection and no need to modify the local virtual element space. For $k\geq 2m$, however, an $L^2$-projection to higher degree polynomial space is needed to control the consistency error; see \S \ref{sec:consistency}.
\end{remark}

With previous preparations, we propose the nonconforming virtual element method for the polyharmonic equation~\eqref{eq:polyharmonic} in any dimension: find $u_h\in V_h$ such that
\begin{equation}\label{polyharmonicNoncfmVEM}
a_h(u_h, v_h)= \langle f, v_h \rangle \quad \forall~v_h\in V_h.
\end{equation}

In the rest of this section, we shall prove the well-posedness of the discretization \eqref{polyharmonicNoncfmVEM} by establishing the coercivity and continuity of the bilinear form $a_h(\cdot, \cdot)$.

\subsection{Mesh conditions}\label{subsec:meshcontitions}
We impose the following conditions on the mesh $\mathcal T_h$.
\begin{itemize}
 \item[(A1)] Each element $K\in \mathcal T_h$ and each face $F\in \mathcal F_h^r$ for $1\leq r\leq n-1$ is star-shaped with a uniformly bounded star-shaped constant.

 \item[(A2)] There exists a quasi-uniform simplicial mesh $\mathcal T_h^*$ such that each $K\in \mathcal T_h$ is a union of some simplexes in $\mathcal T_h^*$.
\end{itemize}

Notice that (A1) and (A2) imply ${\rm diam}(F) \eqsim {\rm diam} (K)$ for all $F\in \mathcal F^r(K), 1\leq r\leq n-1$.

For a star-shaped domain $D$, there exists a ball $B_D\subset D$ with radius $\rho_D h_{D}$ and a Lipschitz isomorphism $\Phi: B_D \to D$ such that $|\Phi|_{1,\infty, B_D}$ and $|\Phi^{-1}|_{1,\infty,D}$ are bounded by a constant depending only on the chunkiness parameter $\rho_D$. Then several trace inequalities of $H^1(D)$ can be established with a constant depending only on $\rho_D$ \cite[(2.18)]{BrennerSung2018}. In particular, we shall use 
\begin{equation}\label{L2trace}
\|v\|_{0,\partial D}^2 \lesssim h_D^{-1}\|v\|_{0,D}^2 + h_D |v|_{1,D}^2 \quad \forall~v\in H^1(D).
\end{equation}


The condition (A2) is inspired by the virtual triangulation condition used in~\cite{ChenHuang2018,BrezziBuffaLipnikov2009}. The simplicial mesh $\mathcal T_h^*$ will serve as a bridge to transfer the results from finite element methods to virtual element methods.

Very recently, some geometric assumptions being the relaxation of conditions (A1)-(A2) were suggested in \cite{CaoChen2018,CaoChen2019} under which a refined error analysis was developed for the linear conforming and nonconforming virtual element methods of the Poisson equation, i.e. $H^1$ case. For high order $H^m, m>1$ elements, we will investigate such relaxation in future works.

\subsection{Weak continuity}
Based on Lemma~\ref{lem:bimgunisolvence}, the space $V_h$ has the weak continuity, that is for any $F\in\mathcal F_h^{1}$, $v_h\in V_h$ and nonnegative integer $s\leq m-1$
\begin{align}
\label{eq:weakcontinuitygradient}
(\llbracket\nabla_h^sv_h\rrbracket, \tau)_F &=0 \quad\forall~\tau\in \mathbb P_{k-(2m-1-s)}(F; \mathbb T_{n}(s)),\\
\label{eq:weakcontinuitygradient2}
Q_0^e(\llbracket\nabla_h^sv_h\rrbracket|_F) &=0 \quad\forall~e\in\mathcal F^{m-s-1}(F),
\end{align}
where $\nabla_h$ is the elementwise gradient with respect to the partition $\mathcal T_h$.
We shall derive some bound on the jump $\llbracket\nabla_h^sv_h\rrbracket$ using the weak continuity and the trace inequality.

By the weak continuity \eqref{eq:weakcontinuitygradient},
the mean value of $\nabla_h^sv_h$ over $F$ is continuous only when $s\geq 2m-1-k$.
For $s<2m-1-k$, the mean value of $\nabla_h^sv_h$ is merely continuous over some low-dimensional face of $F$, cf. \eqref{eq:weakcontinuitygradient2}.
As a concrete example, consider the Morley element in three dimensions. The mean value of $\nabla_hv_h$ over faces is continuous, but the mean value of $v_h$ is only continuous on edges rather than over faces.

Recall the following error estimates of the $L^2$ projection.
\begin{lemma}
Let $\ell\in\mathbb N$. For each $K\in\mathcal T_h$ and $\mathcal{F}^1(K)$, we have for any $v\in H^{\ell+1}(K)$ 
\begin{align}
\label{eq:PKerror}
\|v-Q_{\ell}^Kv\|_{0,K}&\lesssim h_K^{\ell+1}|v|_{\ell+1, K},\\
\label{eq:PFerror}
\|v-Q_{\ell}^Fv\|_{0,F}&\lesssim h_K^{\ell+1/2}|v|_{\ell+1, K}.
\end{align}
\end{lemma}

Then recall the Bramble-Hilbert Lemma (cf.~\cite[Lemma~4.3.8]{BrennerScott2008}).
\begin{lemma}
Let $\ell\in\mathbb N$ and $K\in\mathcal T_h\cup \mathcal T_h^*$. There exists a linear operator $T_{\ell}^K: L^1(K)\to \mathbb P_{\ell}(K)$ such that for any $v\in H^{\ell+1}(K)$,
\begin{equation}\label{eq:Bramble-Hilbert}
\|v-T_{\ell}^Kv\|_{j,K}\lesssim h_K^{\ell+1-j}|v|_{\ell+1, K}\quad\textrm{ for  }\; 0\leq j\leq \ell+1.
\end{equation}
\end{lemma}

Notice that the constants in \eqref{eq:PKerror}-\eqref{eq:Bramble-Hilbert} depend on the star-shaped constant, i.e.
the chunkiness parameter $\rho_K$, and also depend on the degree $\ell$. 

Similarly, define $T_h: L^2(\Omega)\to \mathbb P_k(\mathcal T_h)$ as
\[
(T_hv)|_K:= T_k^{K}(v|_K)\quad\forall~K\in\mathcal T_h.
\]

\begin{lemma}\label{lem:20181013-1}
Given $F\in\mathcal F_h^{1}$ and positive integer $s<m$. Assume for any $e\in \mathcal F^r(F)$ with $r=0,1,\cdots,m-1-s$ %
\begin{equation}\label{eq:20181013-1}
\big\|\llbracket\nabla_h^sv_h\rrbracket\big\|_{0,e}
\lesssim \sum_{K\in\partial^{-1}F}h_{K}^{m-s-(r+1)/2}|v_h|_{m,K}\quad\forall~v_h\in V_h.
\end{equation}
Then we have for any $e\in \mathcal F^r(F)$ with $r=0,1,\cdots,m-s$
\begin{equation}\label{eq:20181013-2}
\big\|\llbracket\nabla_h^{s-1}v_h\rrbracket\big\|_{0,e}
\lesssim \sum_{K\in\partial^{-1}F}h_{K}^{m-s-(r-1)/2}|v_h|_{m,K}\quad\forall~v_h\in V_h.
\end{equation}
\end{lemma}
\begin{proof}
We use the mathematical induction on $r$ to prove \eqref{eq:20181013-2}. First consider $r=0$.
Take some $e_i\in \mathcal F^i(F)$ for $i=1,\cdots, m-s$ such that $e_i\in\mathcal F^1(e_{i-1})$ with $e_0=F$. In the following we shall use $Q_0^e$ the $L^2$-orthogonal projection onto the constant tensor space on $e$ which can be understood as a tensor defined on the whole space.
Employing the trace inequality \eqref{L2trace}, we get from \eqref{eq:20181013-1} with $r=i-1$
\begin{align*}
&h_F^{i/2}\big\|\llbracket\nabla_h^{s-1}v_h\rrbracket-Q_0^{F}(\llbracket\nabla_h^{s-1}v_h\rrbracket)\big\|_{0,e_{i}}
\\
\lesssim & h_F^{(i-1)/2}\big\|\llbracket\nabla_h^{s-1}v_h\rrbracket-Q_0^{F}(\llbracket\nabla_h^{s-1}v_h\rrbracket)\big\|_{0,e_{i-1}} +
h_F^{(i+1)/2}\big\|\llbracket\nabla_h^{s}v_h\rrbracket\big\|_{0,e_{i-1}} \\
\lesssim & h_F^{(i-1)/2}\big\|\llbracket\nabla_h^{s-1}v_h\rrbracket-Q_0^{F}(\llbracket\nabla_h^{s-1}v_h\rrbracket)\big\|_{0,e_{i-1}} +
\sum_{K\in\partial^{-1}F}h_{K}^{m-s+1/2}|v_h|_{m,K}.
\end{align*}
By this recursive inequality and the approximation properties of the $L^2$ projection, it holds
\begin{align*}
&h_F^{(m-s)/2}\big\|\llbracket\nabla_h^{s-1}v_h\rrbracket-Q_0^{F}(\llbracket\nabla_h^{s-1}v_h\rrbracket)\big\|_{0,e_{m-s}}
\\
\lesssim & \big\|\llbracket\nabla_h^{s-1}v_h\rrbracket-Q_0^{F}(\llbracket\nabla_h^{s-1}v_h\rrbracket)\big\|_{0,F} +
\sum_{K\in\partial^{-1}F}h_{K}^{m-s+1/2}|v_h|_{m,K} \\
\lesssim & h_F\big\|\llbracket\nabla_h^{s}v_h\rrbracket\big\|_{0,F} +
\sum_{K\in\partial^{-1}F}h_{K}^{m-s+1/2}|v_h|_{m,K}.
\end{align*}
On the other side, we have from \eqref{eq:weakcontinuitygradient2}
\begin{align*}
&\big\|\llbracket\nabla_h^{s-1}v_h\rrbracket\big\|_{0,F}
=\big\|\llbracket\nabla_h^{s-1}v_h\rrbracket-Q_0^{e_{m-s}}(\llbracket\nabla_h^{s-1}v_h\rrbracket)\big\|_{0,F}\\
=&\big\|\llbracket\nabla_h^{s-1}v_h\rrbracket-Q_0^{F}(\llbracket\nabla_h^{s-1}v_h\rrbracket)-Q_0^{e_{m-s}}(\llbracket\nabla_h^{s-1}v_h\rrbracket- Q_0^{F}(\llbracket\nabla_h^{s-1}v_h\rrbracket))\big\|_{0,F}\\
\leq& \big\|\llbracket\nabla_h^{s-1}v_h\rrbracket-Q_0^{F}(\llbracket\nabla_h^{s-1}v_h\rrbracket)\big\|_{0,F} + \big\|Q_0^{e_{m-s}}(\llbracket\nabla_h^{s-1}v_h\rrbracket-Q_0^{F}(\llbracket\nabla_h^{s-1}v_h\rrbracket))\big\|_{0,F} \\
\lesssim & h_F\big\|\llbracket\nabla_h^{s}v_h\rrbracket\big\|_{0,F} + h_F^{(m-s)/2}\big\|Q_0^{e_{m-s}}(\llbracket\nabla_h^{s-1}v_h\rrbracket-Q_0^{F}(\llbracket\nabla_h^{s-1}v_h\rrbracket))\big\|_{0,e_{m-s}} \\
\leq & h_F\big\|\llbracket\nabla_h^{s}v_h\rrbracket\big\|_{0,F} + h_F^{(m-s)/2}\big\|\llbracket\nabla_h^{s-1}v_h\rrbracket-Q_0^{F}(\llbracket\nabla_h^{s-1}v_h\rrbracket)\big\|_{0,e_{m-s}}.
\end{align*}
Hence \eqref{eq:20181013-2} with $r=0$ follows from the last two inequalities and \eqref{eq:20181013-1} with $r=0$.

Assume \eqref{eq:20181013-2} holds with $r=j<m-s$. Let $e\in \mathcal F^{j+1}(F)$.
Take some $e_j\in \mathcal F^{j}(F)$ satisfying $e\in\mathcal F^{1}(e_j)$. Using the trace inequality \eqref{L2trace} again, we know
\[
\big\|\llbracket\nabla_h^{s-1}v_h\rrbracket\big\|_{0,e}\lesssim h_e^{-1/2}\big\|\llbracket\nabla_h^{s-1}v_h\rrbracket\big\|_{0,e_j} + h_e^{1/2}\big\|\llbracket\nabla_h^{s}v_h\rrbracket\big\|_{0,e_j},
\]
which combined with the assumptions means \eqref{eq:20181013-2} is true with $r=j+1$.
\end{proof}

Again consider the Morley-Wang-Xu element in three dimensions \cite{WangXu2006}, i.e. $m=2$ and $s=1$.
The inequality \eqref{eq:20181013-1} is just
\[
\big\|\llbracket\nabla_hv_h\rrbracket\big\|_{0,F}
\lesssim \sum_{K\in\partial^{-1}F}h_{K}^{1/2}|v_h|_{2,K}\quad\forall~v_h\in V_h,
\]
for each face $F\in\mathcal F_h^{1}$.
Then by Lemma~\ref{lem:20181013-1} we will get from \eqref{eq:weakcontinuitygradient2}
\[
\big\|\llbracket v_h\rrbracket\big\|_{0,F}
+h_{K}^{1/2}\big\|\llbracket v_h\rrbracket\big\|_{0,e}\lesssim \sum_{K\in\partial^{-1}F}h_{K}^{3/2}|v_h|_{2,K}\quad\forall~v_h\in V_h,
\]
for any face $F\in\mathcal F_h^{1}$ and any edge $e\in\mathcal F_h^{2}$.
We refer to \cite[Lemma~5]{WangXu2006} for these estimates on tetrahedra.

\begin{lemma}
For each $F\in\mathcal F_h^{1}$ and nonnegative integer $s<m$, it holds %
\begin{equation}\label{eq:weakcontinuityestimate}
\big\|\llbracket\nabla_h^sv_h\rrbracket\big\|_{0,F}
\lesssim \sum_{K\in\partial^{-1}F}h_{K}^{m-s-1/2}|v_h|_{m,K}\quad\forall~v_h\in V_h.
\end{equation}
\end{lemma}
\begin{proof}
It is sufficient to prove that for $s=m-1,m-2,\cdots,0$ and any $e\in \mathcal F^r(F)$ with $r=0,1,\cdots,m-1-s$, it hold
\[
\big\|\llbracket\nabla_h^sv_h\rrbracket\big\|_{0,e}
\lesssim \sum_{K\in\partial^{-1}F}h_{K}^{m-s-(r+1)/2}|v_h|_{m,K}\quad\forall~v_h\in V_h.
\]
According to Lemma~\ref{lem:20181013-1} and the mathematical induction, we only need to show
\[
\big\|\llbracket\nabla_h^{m-1}v_h\rrbracket\big\|_{0,F}
\lesssim \sum_{K\in\partial^{-1}F}h_{K}^{1/2}|v_h|_{m,K}\quad\forall~v_h\in V_h.
\]
In fact, due to \eqref{eq:weakcontinuitygradient2} and \eqref{eq:PFerror}, we get
\[
\big\|\llbracket\nabla_h^{m-1}v_h\rrbracket\big\|_{0,F}=\big\|\llbracket\nabla_h^{m-1}v_h\rrbracket-Q_0^F(\llbracket\nabla_h^{m-1}v_h\rrbracket)\big\|_{0,F}\lesssim \sum_{K\in\partial^{-1}F}h_{K}^{1/2}|v_h|_{m,K}.
\]
This ends the proof.
\end{proof}

Given the virtual triangulation $\mathcal T_h^{\ast}$, for each nonnegative integer $r<m$, define the tensorial $(m-r)$-th order Lagrange element space associated with  $\mathcal T_h^{\ast}$
\[
S_h^{m,r}:=\{\tau_h\in H_0^1(\Omega; \mathbb T_{n}(r)): \tau_h|_K\in\mathbb P_{m-r}(G; \mathbb T_{n}(r))\quad\forall~K\in\mathcal T_h^{\ast}\}.
\]

\begin{lemma}\label{lem:connectionerror}
Let $r=0,1,\cdots, m-1$.
For any $v_h\in V_h$, there exists $\tau_r= \tau_r (v_h)\in S_h^{m,r}$ such that
\begin{equation}\label{eq:connectionerror}
|\nabla_h^rv_h-\tau_r|_{j,h}\lesssim h^{m-r-j}|v_h|_{m,h}
\quad \textrm{ for }\; j=0, 1, \cdots, m-r.
\end{equation}
\end{lemma}
\begin{proof}
Let $w_h\in L^2(\Omega; \mathbb T_{n}(r))$ be defined as
\[
w_h|_K=T_{m-r-1}^K(\nabla^r(v_h|_K))\quad\forall~K\in\mathcal T_h^{\ast}.
\]
Since $w_h$ is a piecewise tensorial polynomial, by Lemma 3.1 in~\cite{WangXu2013},
there exists $\tau_r\in S_h^{m,r}$ such that
\[
|w_h-\tau_r|_{j,h}^2\lesssim \sum_{F\in\mathcal F_h^1(\mathcal T_h^{\ast})}h_F^{1-2j}\|\llbracket w_h\rrbracket\|_{0,F}^2, 
\]
where $\mathcal F_h^1(\mathcal T_h^{\ast})$ is the set of all $(n-1)$-dimensional faces of the partition $\mathcal T_h^{\ast}$. Then it follows from~\eqref{eq:Bramble-Hilbert} and \eqref{eq:weakcontinuityestimate}
\begin{align*}
|w_h-\tau_r|_{j,h}^2&\lesssim \sum_{F\in\mathcal F_h^1(\mathcal T_h^{\ast})}h_F^{1-2j}\|\llbracket w_h-\nabla_h^rv_h\rrbracket\|_{0,F}^2 + \sum_{F\in\mathcal F_h^1}h_F^{1-2j}\|\llbracket \nabla_h^rv_h\rrbracket\|_{0,F}^2 \\
&\lesssim h^{2(m-r-j)}|v_h|_{m,h}^2.
\end{align*}
Here we have used the fact that the jump $\llbracket \nabla_h^rv_h\rrbracket$ is zero on $F\in \mathcal F_h^1(\mathcal T_h^{\ast})\setminus \mathcal F_h^1(\mathcal T_h)$.
Applying~\eqref{eq:Bramble-Hilbert} again gives
\[
|\nabla_h^rv_h-w_h|_{j,h}\lesssim h^{m-r-j}|v_h|_{m,h}.
\]
Finally combining the last two inequalities indicate~\eqref{eq:connectionerror}.
\end{proof}

\begin{lemma}
We have the discrete Poincar\'e inequality
\begin{equation}\label{eq:poincareinequality}
\|v_h\|_{m,h}\lesssim |v_h|_{m,h}\quad \forall~v_h\in V_h.
\end{equation}
\end{lemma}
\begin{proof}
By picking $\tau_r\in H_0^1(\Omega; \mathbb T_{n}(r))$ as in Lemma~\ref{lem:connectionerror}, due to~\eqref{eq:connectionerror}
and the Poincar\'e inequality,
we have for $r=0,1,\cdots, m-1$,
\begin{align*}
\|\nabla_h^rv_h\|_0&\leq \|\nabla_h^rv_h-\tau_r\|_0+\|\tau_r\|_0\lesssim |v_h|_{m,h}+|\tau_r|_1 \\
&\leq |v_h|_{m,h}+|\nabla_h^rv_h-\tau_r|_{1,h}+|\nabla_h^rv_h|_{1,h} \\
&\lesssim |v_h|_{m,h}+\|\nabla_h^{r+1}v_h\|_0,
\end{align*}
which leads to~\eqref{eq:poincareinequality}.
\end{proof}

The discrete Poincar\'e inequality~\eqref{eq:poincareinequality} means
\[
\|v_h\|_{m,h}\eqsim |v_h|_{m,h}\quad \forall~v_h\in V_h,
\]
i.e. $|\cdot|_{m,h}$ is a norm on the space $V_h$.

\subsection{Norm equivalence and well-posedness of the discretization}

Denote by $\ker(\Pi^K)\subset W_k(K)$ the kernel space of the operator $\Pi^K$.
By~\eqref{eq:projectPk} and Lemma~\ref{eq:dofsunisolventWK}, both $|\cdot|_{m, K}$ and $S_K^{1/2}(\cdot, \cdot)$ are norms on the finite dimensional space $\ker(\Pi^K)$. 
Then we have the following norm equivalence. 
\begin{theorem}\label{th:normequivalence}
Assume the mesh $\mathcal T_h$ satisfies conditions (A1) and (A2). For any $K\in\mathcal T_h$, the following norm equivalence holds
\begin{equation}\label{eq:SKequiv}
S_K(v, v)\eqsim |v|_{m, K}^2\quad\forall~v\in \ker(\Pi^K),
\end{equation}
where the constant is independent of $h_K$, but may depend on the chunkiness parameter $\rho_K$, the degree of polynomials $k$, the order of differentiation $m$, the dimension of space $n$, and the shape regularity and quasi-uniform constants of the virtual triangulation $\mathcal T^*_h$.
\end{theorem}
Using the generalized scaling argument, i.e.,
applying an affine map $\hat {\boldsymbol  x} = (\boldsymbol  x- \boldsymbol
x_K)/h_{K}$, it is easy to show the norm equivalence constant is independent of the diameter of $K$.

The constant in \eqref{eq:SKequiv}, however, could still depend on the geometry of $K$ and a clear dependence is not easy to characterize. For finite element space defined on simplexes, the shape functions are usually polynomials and there exists an affine map to the reference element $\hat K$. The norm equivalence on the reference element can be used. Since the Jacobi matrix is constant, the norm $H^m(K)$ and $H^m(\hat K)$ can be clearly characterized using the geometry of the simplex, e.g., the angles of a triangle in 2D.

Now for a general polytope $K$,  there does not exist an affine-equivalent reference polytope $\hat K$. For star-shaped and Lipschitz continuous domain, one can use the isomorphism $\Phi: K\to B_{K}$ but $\Phi \in W^{1,\infty}(B_K)$ only. One can apply the norm equivalence on $B_K$ but how the norm $H^m(K)$ related to $H^m(B_K)$, for $m>1$, is not clear.

We shall prove the norm equivalence \eqref{eq:SKequiv} with mesh conditions (A1)-(A2) in Appendix \ref{sec:appendixA}.


By the Cauchy-Schwarz inequality and the norm equivalence~\eqref{eq:SKequiv}, we have
\begin{equation}\label{eq:SKbound}
S_K(w, v)\lesssim  |w|_{m, K}|v|_{m, K}\quad\forall~w, v\in \ker(\Pi^K).
\end{equation}
which implies the continuity of $a_h(\cdot,\cdot)$
\begin{equation}\label{eq:ahbound}
a_h(w_h, v_h)\lesssim  |w_h|_{m, h}|v_h|_{m, h}\quad\forall~w_h, v_h\in V_h+\mathbb P_k(\mathcal T_h).
\end{equation}
Next we verify the coercivity of $a_h(\cdot,\cdot)$.
\begin{lemma}
For any $v_h\in V_h+\mathbb P_k(\mathcal T_h)$, it holds
\begin{equation}\label{eq:ahcoercivity}
|v_h|_{m,h}^2\lesssim a_h(v_h, v_h).
\end{equation}
\end{lemma}
\begin{proof}
Since $\Pi^K$ is the $H^m$-orthogonal projection,
\[
|v_h|_{m, K}^2 = \left|\Pi^K(v_h|_K)\right|_{m, K}^2 + \left|v_h-\Pi^K(v_h|_K)\right|_{m, K}^2.
\]
Applying~\eqref{eq:SKequiv}, we have
\begin{align}
|v_h|_{m, K}^2 &\lesssim  \left|\Pi^K(v_h|_K)\right|_{m, K}^2 + S_K(v_h-\Pi^K(v_h|_K), v_h-\Pi^K(v_h|_K))\notag\\
&=a_{h,K}(v_h, v_h), \label{eq:20181012-3}
\end{align}
which implies~\eqref{eq:ahcoercivity}.
\end{proof}

Therefore the nonconforming virtual element method~\eqref{polyharmonicNoncfmVEM} is uniquely solvable by the Lax-Milgram lemma.

\section{Error Analysis}\label{sec:erroranalysis}
In this section, we will develop the error analysis of the nonconforming virtual element method~\eqref{polyharmonicNoncfmVEM} for $H^m$-problem.

\subsection{Interpolation error estimate}
We first explore a discrete Galerkin orthogonality of $u-I_hu$ for the nonconforming element, where $I_h$ defined on $H^m(\Omega)$ is the global canonical interpolation operator based on the degrees of freedom in~\eqref{Hmdof1}-\eqref{Hmdof2}, i.e., $(I_hv)|_K:=I_K(v|_K)$ for any $v\in H^m(\Omega)$ and $K\in\mathcal T_h$. A similar result was given in~\cite[(3.3)]{HuShiXu2012}.
\begin{lemma}\label{lem:20190605}
For each $K\in\mathcal T_h$, any $v\in H^m(\Omega)$ and $w \in H^m(K)$, it holds
\begin{equation}\label{eq:20181012-2}
a_{h,K}(v-I_hv, w)=0.
\end{equation}
\end{lemma}
\begin{proof}
It follows from \eqref{eq:20181012-1} and the definition of $S_K(\cdot, \cdot)$
\begin{align*}
a_{h,K}(v-I_hv, w)&=(\nabla^m\Pi^K(v-I_hv), \nabla^m\Pi^Kw)_K \\
&\quad+S_K(v-I_hv-\Pi^K(v-I_hv), w-\Pi^Kw) \\
&=S_K(v-I_hv, w-\Pi^Kw)=0,
\end{align*}
in the last step we use the fact that $v$ and $I_hv$ share the same degrees of freedom, and thus the stabilization $S_K(v-I_hv, w-\Pi^Kw)$ using d.o.f. vanishes, cf. \eqref{eq:stabilization}.
\end{proof}

\begin{remark}\rm
Lemma~\ref{lem:20190605} holds true in virtue of the choice \eqref{eq:stabilization}; indeed, any stabilization
equivalent to \eqref{eq:stabilization} which annihilates if all the degrees of freedom are zero would be fine.
\end{remark}

With the help of the discrete Galerkin orthogonality, we present the following interpolation error estimate.
\begin{lemma}
For each $K\in\mathcal T_h$ and any $v\in H^{k+1}(K)$, we have
\begin{equation}\label{eq:Iherror}
|v-I_hv|_{m,K}\lesssim h_K^{k+1-m}|v|_{k+1, K}.
\end{equation}
\end{lemma}
\begin{proof}
Applying \eqref{eq:20181012-3} and \eqref{eq:20181012-2} with $w=(T_hv-I_hv)|_K$, we have
\begin{align*}
|T_hv-I_hv|_{m,K}^2&\lesssim a_{h,K}(T_hv-I_hv, T_hv-I_hv)=a_{h,K}(T_hv-v, T_hv-I_hv)\\
&\lesssim |v-T_hv|_{m,K}|T_hv-I_hv|_{m,K},
\end{align*}
which indicates
\[
|T_hv-I_hv|_{m,K}\lesssim |v-T_hv|_{m,K}.
\]
Hence
\[
|v-I_hv|_{m,K}\leq |v-T_hv|_{m,K}+|T_hv-I_hv|_{m,K}\lesssim |v-T_hv|_{m,K}.
\]
Therefore \eqref{eq:Iherror} follows from \eqref{eq:Bramble-Hilbert}.
\end{proof}

\subsection{Consistency error estimate}\label{sec:consistency}
Due to~\eqref{eq:projectPk} and~\eqref{eq:H2projlocal1}, we have the following $k$-consistency.
\begin{lemma}
For any $p\in \mathbb P_k(K)$ and any $v\in W_k^K$, it holds
\begin{equation}\label{eq:kconsistency}
a_{h,K}(p, v)=(\nabla^mp, \nabla^mv)_K.
\end{equation}
\end{lemma}

To estimate the consistency error of the discretization, we split it into two cases, i.e. $k\geq 2m-1$ and $m\leq k< 2m-1$.
For the first case,
the weak continuity \eqref{eq:weakcontinuitygradient}, that is the projection $Q_{k-(m+i)}^F(\nabla_h^{m-(i+1)}v_h)$ is continuous across $F\in\mathcal{F}_h^1$ for $i=0,\cdots,m-1$, is sufficient to derive the optimal consistency error estimate.

\begin{lemma}\label{lem:consistencyerror1}
Let $u\in H_0^m(\Omega)\cap H^{k+1}(\Omega)$ be the solution of the polyharmonic equation~\eqref{eq:polyharmonic}. Assume $k\geq 2m-1$. Then it holds
\begin{equation}\label{eq:consistencyerror1}
(\nabla^mu, \nabla_h^mv_h)-(f, v_h)\lesssim h^{k+1-m}|u|_{k+1}|v_h|_{m,h}\quad\forall~v_h\in V_h.
\end{equation}
\end{lemma}
\begin{proof}
First we notice that
\begin{align}
&(\nabla^mu, \nabla_h^mv_h)-(f, v_h)\notag\\
=&\sum_{i=0}^{m-1}(-1)^i\Big ((\div^i\nabla^mu, \nabla_h^{m-i}v_h)+(\div^{i+1}\nabla^mu, \nabla_h^{m-(i+1)}v_h)\Big ).\label{eq:temp20180808-1}
\end{align}
For each term in the right hand side of~\eqref{eq:temp20180808-1}, applying integration by parts,~\eqref{eq:weakcontinuitygradient} with $s=m-(i+1)$,~\eqref{eq:PFerror} and~\eqref{eq:weakcontinuityestimate}, we get
\begin{align*}
&(-1)^i\Big ((\div^i\nabla^mu, \nabla_h^{m-i}v_h)+(\div^{i+1}\nabla^mu, \nabla_h^{m-(i+1)}v_h)\Big ) \\
=&(-1)^i\sum_{K\in\mathcal T_h}((\div^i\nabla^mu)\nu, \nabla_h^{m-(i+1)}v_h)_{\partial K} \\
=&(-1)^i\sum_{F\in\mathcal{F}_h^1}((\div^i\nabla^mu)\nu_{F,1}, \llbracket\nabla_h^{m-(i+1)}v_h\rrbracket)_{F} \\
=&(-1)^i\sum_{F\in\mathcal{F}_h^1}((\div^i\nabla^mu)\nu_{F,1}-Q_{k-(m+i)}^F((\div^i\nabla^mu)\nu_{F,1}), \llbracket\nabla_h^{m-(i+1)}v_h\rrbracket)_{F} \\
\lesssim & h^{k+1-m}|u|_{k+1}|v_h|_{m,h},
\end{align*}
as required.
\end{proof}

When the order $k$ is not high enough, the mean value of $\nabla_h^sv_h$ is only continuous over some low-dimensional face of $F$ for $s<2m-1-k$. In this case, we divide the consistency error into two parts. The first part is estimated by using the weak continuity \eqref{eq:weakcontinuitygradient} as \textcolor[rgb]{1.00,0.00,0.00}{in} the proof of \eqref{eq:consistencyerror1}, while the second part is estimated by using  the weak continuity \eqref{eq:weakcontinuitygradient2} through employing the Lagrange element space as a bridge. 

\begin{lemma}
Let $u\in H_0^m(\Omega)\cap H^{2m-1}(\Omega)$ be the solution of the polyharmonic equation~\eqref{eq:polyharmonic}. Assume $m\leq k< 2m-1$. Then it holds
\begin{equation}\label{eq:consistencyerror2}
(\nabla^mu, \nabla_h^mv_h)-(f, v_h)\lesssim \Big (\sum_{i=k+1-m}^{m-1}h^i|u|_{m+i} + h^{m}\|f\|_{0}\Big )|v_h|_{m,h} \;\;\forall~v_h\in V_h.
\end{equation}
\end{lemma}
\begin{proof}
Similarly as~\eqref{eq:temp20180808-1}, we have
\begin{equation}\label{eq:temp20180808-2}
(\nabla^mu, \nabla_h^mv_h)-(f, v_h)=E_1+E_2+E_3,
\end{equation}
where
\begin{align*}
E_1&:=\sum_{i=0}^{k-m}(-1)^i\Big ((\div^i\nabla^mu, \nabla_h^{m-i}v_h)+(\div^{i+1}\nabla^mu, \nabla_h^{m-(i+1)}v_h)\Big ),
\\
E_2&:=\sum_{i=k-m+1}^{m-2}(-1)^i\Big ((\div^i\nabla^mu, \nabla_h^{m-i}v_h)+(\div^{i+1}\nabla^mu, \nabla_h^{m-(i+1)}v_h)\Big ),
\\
E_3&:=((-\div)^{m-1}\nabla^mu, \nabla_hv_h)-(f, v_h).
\end{align*}
By the same argument as in the proof of Lemma~\ref{lem:consistencyerror1}, we have
\begin{equation}\label{eq:E1}
E_1\lesssim h^{k+1-m}|u|_{k+1}|v_h|_{m,h}.
\end{equation}
Next let us estimate $E_2$ and $E_3$. By~\eqref{eq:connectionerror}, for each $k-m+1\leq i\leq m-1$, there exists $\tau_{m-(i+1)}\in S_h^{m, m-(i+1)}$ such that
\begin{equation}\label{eq:temp20180809-1}
|\nabla_h^{m-(i+1)}v_h-\tau_{m-(i+1)}|_{j,h}\lesssim h^{i+1-j}|v_h|_{m,h}
\quad \textrm{ for }\; j=0, 1.
\end{equation}
Since $\tau_{m-(i+1)}\in H_0^1(\Omega; \mathbb T_{n}(m-(i+1)))$, we get for $i=k-m+1,\cdots, m-2$
\begin{equation}\label{eq:temp20180809-2}
(\div^i\nabla^mu, \nabla\tau_{m-(i+1)})+(\div^{i+1}\nabla^mu, \tau_{m-(i+1)})=0,
\end{equation}
\begin{equation}\label{eq:temp20180809-3}
((-\div)^{m-1}\nabla^mu, \nabla\tau_{0})-(f, \tau_{0})=0.
\end{equation}
For $i=k-m+1,\cdots, m-2$, it follows from~\eqref{eq:temp20180809-1}-\eqref{eq:temp20180809-2}
\begin{align*}
&\,(-1)^i\Big ((\div^i\nabla^mu, \nabla_h^{m-i}v_h)+(\div^{i+1}\nabla^mu, \nabla_h^{m-(i+1)}v_h)\Big ) \\
=&\, (-1)^i(\div^i\nabla^mu, \nabla_h(\nabla_h^{m-(i+1)}v_h-\tau_{m-(i+1)})) \\
&\,+(-1)^i(\div^{i+1}\nabla^mu, \nabla_h^{m-(i+1)}v_h-\tau_{m-(i+1)})  \\
\lesssim &\, h^i|u|_{m+i}|v_h|_{m,h} + h^{i+1}|u|_{m+i+1}|v_h|_{m,h}.
\end{align*}
Thus we obtain
\begin{equation}\label{eq:E2}
E_2\lesssim \sum_{i=k+1-m}^{m-1}h^i|u|_{m+i} |v_h|_{m,h}.
\end{equation}
Similarly, employing~\eqref{eq:temp20180809-3} and~\eqref{eq:temp20180809-1}, we get
\begin{align*}
E_3&=((-\div)^{m-1}\nabla^mu, \nabla_hv_h)-(f, v_h)\\
&=((-\div)^{m-1}\nabla^mu, \nabla_h(v_h-\tau_{0}))-(f, v_h-\tau_{0})\\
&\lesssim h^{m-1}|u|_{2m-1}|v_h|_{m,h} + h^{m}\|f\|_{0}|v_h|_{m,h},
\end{align*}
which together with~\eqref{eq:temp20180808-2}-\eqref{eq:E1} and~\eqref{eq:E2} ends the proof.
\end{proof}

We then consider the perturbation of the right hand side. Namely replace the $L^2$-inner product $(f, v_h)$ by an approximated one $\langle f, v_h \rangle$ defined in \eqref{eq:fvh}.
\begin{lemma}\label{lem:consistencyerror}
Let $u\in H_0^m(\Omega)\cap H^{r}(\Omega)$ with $r=\max\{k+1, 2m-1\}$ be the solution of the polyharmonic equation~\eqref{eq:polyharmonic}.
Assume $f\in H^{\ell}(\mathcal T_h)$ with $\ell=\max\{0,k+1-2m\}$. Then it holds for any $v_h\in V_h$
\begin{align}
&(\nabla^mu, \nabla_h^mv_h)-\langle f, v_h \rangle\notag\\
\lesssim& h^{k+1-m}(\|u\|_{r}+h\|f\|_0+h^{\max\{0,2m-k-1\}}|f|_{\ell,h})|v_h|_{m,h}.\label{eq:consistencyerror}
\end{align}
\end{lemma}
\begin{proof}
It follows from~\eqref{eq:consistencyerror1} and~\eqref{eq:consistencyerror2}
\[
(\nabla^mu, \nabla_h^mv_h)-(f, v_h)\lesssim h^{k+1-m}(\|u\|_{r}+h\|f\|_0)|v_h|_{m,h}.
\]
For $m\leq k\leq 2m-1$, we get from the local Poincar\'e inequality \eqref{eq:poincare}
\[
(f, v_h)-\langle f, v_h \rangle=(f, v_h-\Pi_hv_h)\lesssim h^{m}\|f\|_{0}|v_h|_{m,h}.
\]
For $k\geq 2m$, it holds from \eqref{eq:PKerror}
\begin{align*}
(f, v_h)-\langle f, v_h \rangle&=\Big(f, v_h-Q_h^{\max\{m-1, k-2m\}}v_h\Big)\\
&=\Big(f-Q_h^{k-2m}f, v_h-Q_h^{\max\{m-1, k-2m\}}v_h\Big)\\
&\leq\|f-Q_h^{k-2m}f\|_0\|v_h-Q_h^{m-1}v_h\|_0\\
&\lesssim h^{k+1-m}|f|_{k+1-2m,h}|v_h|_{m,h}.
\end{align*}
Thus we conclude~\eqref{eq:consistencyerror} from the last three inequalities.
\end{proof}

\subsection{Error estimate}
Now we are in a position to present the optimal order convergence of our nonconforming virtual element method.
\begin{theorem}
Let $u\in H_0^m(\Omega)\cap H^{r}(\Omega)$ with $r=\max\{k+1, 2m-1\}$ be the solution of the polyharmonic equation~\eqref{eq:polyharmonic}, and $u_h\in V_h$ be the nonconforming virtual element method~\eqref{polyharmonicNoncfmVEM}. Assume the mesh $\mathcal T_h$ satisfies conditions (A1) and (A2). Assume $f\in H^{\ell}(\mathcal T_h)$ with $\ell=\max\{0,k+1-2m\}$.
Then it holds
\begin{equation}\label{eq:energyerror}
|u-u_h|_{m,h}\lesssim h^{k+1-m}(\|u\|_{r}+h\|f\|_0+h^{\max\{0,2m-k-1\}}|f|_{\ell,h}).
\end{equation}
\end{theorem}
\begin{proof}
Let $v_h=I_hu-u_h$. From~\eqref{eq:ahbound},~\eqref{eq:Iherror} and~\eqref{eq:Bramble-Hilbert}, it holds
\begin{align}
&a_h(I_hu-T_hu, v_h)+(\nabla_h^m(T_hu-u), \nabla_h^mv_h) \notag\\
\lesssim & |I_hu-T_hu|_{m,h}|v_h|_{m,h}+|u-T_hu|_{m,h}|v_h|_{m,h} \notag\\
\lesssim & (|u-I_hu|_{m,h}+|u-T_hu|_{m,h})|v_h|_{m,h} \lesssim  h^{k+1-m}|u|_{k+1}|v_h|_{m,h}. \label{eq:temp201806181}
\end{align}
Employing~\eqref{eq:ahcoercivity},~\eqref{polyharmonicNoncfmVEM} and~\eqref{eq:kconsistency}, we have
\begin{align*}
|I_hu-u_h|_{m,h}^2&\lesssim a_h(I_hu-u_h, v_h)=a_h(I_hu, v_h)- \langle f, v_h \rangle \\
&=a_h(I_hu-T_hu, v_h)+a_h(T_hu, v_h)-\langle f, v_h \rangle \\
&=a_h(I_hu-T_hu, v_h)+(\nabla_h^mT_hu, \nabla_h^mv_h)-\langle f, v_h \rangle \\
&=a_h(I_hu-T_hu, v_h)+(\nabla_h^m(T_hu-u), \nabla_h^mv_h) \\
&\quad\;+(\nabla^mu, \nabla_h^mv_h)-\langle f, v_h \rangle.
\end{align*}
Then we get from~\eqref{eq:temp201806181} and~\eqref{eq:consistencyerror}
\[
|I_hu-u_h|_{m,h}\lesssim h^{k+1-m}(\|u\|_{r}+h\|f\|_0+h^{\max\{0,2m-k-1\}}|f|_{\ell,h}).
\]
Finally we derive~\eqref{eq:energyerror} by combining the last inequality and~\eqref{eq:Iherror}.
\end{proof}

\section{Conclusion}

In view of a generalized Green's identity for $H^m$ inner product, we construct the $H^m$-nonconforming virtual elements of any order $k$ on any shape of polytope in $\mathbb R^n$ with constraints $m\leq n$ and $k\geq m$ in a unified way in this paper.
A rigorous and detailed convergence analysis is developed for the $H^m$-nonconforming virtual element methods, and the optimal error estimates are achieved.
When $m>n$, the generalized Green's identity for $H^m$ inner product, the key tool in this paper, will involve the derivative terms on zero-dimensional subsurfaces, i.e., nodes of the mesh. We will postpone the case $m>n$ in future works.

This paper is motivated by the theoretical purposes.
The numerical investigation of the virtual element method proposed in this paper is postponed to future works.

\section*{Acknowledgements}

The authors would like to thank Dr. Shuhao Cao for the insightful discussion and the anonymous referees for their incisive comments.

 \appendix
 \section{Norm equivalence}\label{sec:appendixA}

As mentioned before, it is difficult to derive the norm equivalence~\eqref{eq:SKequiv} directly from the norm equivalence of the finite-dimensional space due to the absence of an affine-equivalent reference polytope. We shall prove the norm equivalence~\eqref{eq:SKequiv} in this appendix by assuming that the mesh $\mathcal T_h$ satisfies conditions (A1) and (A2). For $m=1$, proofs on the norm equivalence for $H^1$ conforming VEM can be found in~\cite{Beirao-da-Veiga;Lovadina;Russo:2017Stability,BrennerSung2018,ChenHuang2018}.

With the help of the virtual triangulation $\mathcal T_h^*$, we can prove the inverse inequality of polynomial spaces on $K$ following the proof in \cite[Lemma 3.1]{ChenHuang2018}
\begin{equation}\label{eq:polyinverse}
 \| g \|_{0,K} \lesssim h_K^{-i}\| g \|_{-i,K} \quad \forall~g\in \mathbb P_k(K), \;i = 1,2,\ldots, m,
\end{equation}
where the constant depends only on the degree of polynomials $k$, the order of differentiation $m$, the dimension of space $n$, and the shape regularity and quasi-uniformity of the virtual triangulation $\mathcal T_h^*(K)$. 

On the polynomial space, we have the normal equivalence of the $L^2$-norm of $g$ and $l^2$-norm of its d.o.f. Let $g = \sum_{i} g_{i}m_{i}$ be a polynomial on $F$, where $F\in\mathcal{F}^j(K)$ with $j\geq 1$. Denote by $\boldsymbol  g = (g_{i})$ the coefficient vector. Then the following norm equivalence holds (cf. \cite[Lemma~4.1]{ChenHuang2018})
\begin{equation}\label{eq:polynorm}
h_{F}^{(n-j)/2}\|\boldsymbol  g\|_{l^2} \lesssim \|g\|_{0,F}\lesssim h_F^{(n-j)/2}\|\boldsymbol  g\|_{l^2}.
\end{equation}


Take an element $K\in\mathcal T_h$. For any $F\in\mathcal{F}^j(K)$ with $j\geq 1$, let $\mathbb R_F^{n-j}$ be the $(n-j)$-dimensional affine space passing through $F$, $\mathcal{F}_F^j(K):=\{e\in\mathcal{F}^j(K): e\subset\mathbb R_F^{n-j}\}$, and
$$
\lambda_{F,i}:=\nu_{F,i}^{\intercal}\frac{\boldsymbol x-\boldsymbol x_F}{h_K},\quad i=1,\cdots, j.
$$
Apparently $\lambda_{F,i}|_F=0$, i.e. the $(n-1)$-dimensional face $\lambda_{F,i}=0$ passes through $F$.
If $K$ is a simplex and $F\in \mathcal F^{1}(K)$, $\lambda_{F,1}$ is just the barycenter coordinate when $h_F$ represents the height of $K$ corresponding to the base $F$.
And $\mathcal{F}_F^j(K)=\{F\}$ if $K$ is strictly convex.
For any $F\in\mathcal{F}^j(K)$ with $j\geq 1$, and $F'\in\mathcal{F}^j(K)\backslash\mathcal{F}_F^j(K)$, let $\nu_{F,F'}$ be some unit normal vector of $F'$ such that the hyperplane $\nu_{F,F'}^{\intercal}(\bs x-\bs x_{F'})=0$ does not pass through $F$.
Define bubble functions
\[
b_K:=\prod_{F\in\mathcal F^{1}(K)}\lambda_{F,1},
\]
$$
b_F:=\bigg(\prod_{F'\in\mathcal{F}^j(K)\backslash\mathcal{F}_F^j(K)}\nu_{F,F'}^{\intercal}\frac{\boldsymbol x-\boldsymbol x_{F'}}{h_K}\bigg)\bigg(\prod_{F'\in\mathcal{F}_F^j(K)}\prod_{e\in\mathcal F^{1}(F')}\nu_{F',e}^{\intercal}\frac{\boldsymbol x-\boldsymbol x_{e}}{h_K}\bigg),
$$
for each $F\in \mathcal F^{j}(K)$ with $1\leq j\leq n$. Notice that both $b_K$ and $b_F$ are polynomials.

\begin{lemma}
Let $K\in\mathcal T_h$. It holds
\begin{equation}\label{eq:inverseeq1}
h_K^m\|(-\Delta)^mv\|_{0,K}\lesssim \|\nabla^mv\|_{0,K}\quad\forall~v\in V_k(K)\cup W_k(K).
\end{equation}
\end{lemma}
\begin{proof}
Let $\phi_K:=b_K^{2m}(-\Delta)^mv\in H_0^m(K)$, then $\|\phi_K\|_{0,K}\eqsim \|(-\Delta)^mv\|_{0,K}$. Using the scaling argument, the integration by parts and the inverse inequality for polynomials \eqref{eq:polyinverse}, we get
\begin{align*}
\|(-\Delta)^mv\|_{0,K}^2&\lesssim ((-\Delta)^mv, \phi_K)_K= (\nabla^mv, \nabla^m\phi_K)_K \\
&\leq \|\nabla^mv\|_{0,K}\|\nabla^m\phi_K\|_{0,K}\lesssim h_K^{-m}\|\nabla^mv\|_{0,K}\|\phi_K\|_{0,K} \\
&\lesssim h_K^{-m}\|\nabla^mv\|_{0,K}\|(-\Delta)^mv\|_{0,K},
\end{align*}
which induces the required inequality.
\end{proof}

\begin{lemma}
Let $K\in\mathcal T_h$.
For any positive integer $j\leq m$, $F\in\mathcal F^{j}(K)$, and $\alpha\in A_{j}$ with $|\alpha|\leq m-j$, we have
\begin{align}
&\sum_{F'\in\mathcal{F}_F^j(K)}h_K^{m-|\alpha|-j/2}\|D^{2m-j-|\alpha|}_{F', \alpha}(v)\|_{0,F'}\notag\\
\lesssim &\|\nabla^mv\|_{0,K} + h_K^{m}\|(-\Delta)^mv\|_{0,K}\notag\\
&+ \sum_{\ell=1}^{j-1}\sum_{e\in\mathcal F^{\ell}(K)}\sum_{\beta\in A_{\ell}\atop|\beta|\leq m-\ell}h_K^{m-|\beta|-\ell/2}\Big \|D^{2m-\ell-|\beta|}_{e, \beta}(v)\Big\|_{0,e} \notag\\
&+ \sum_{F'\in\mathcal{F}_F^j(K)}\sum_{\beta\in A_{j}\atop|\alpha|<|\beta|\leq m-j}h_K^{m-|\beta|-\ell/2}\Big \|D^{2m-j-|\beta|}_{F', \beta}(v)\Big\|_{0,F'}\label{eq:inverseeq2}
\end{align}
for all $v\in V_k(K)\cup W_k(K)$.
\end{lemma}
\begin{proof}
Since $D^{2m-j-|\alpha|}_{F', \alpha}(v)|_{F'}$ is a polynomial for each $F'\in \mathcal{F}_F^j(K)$, we can regard $D^{2m-j-|\alpha|}_{F', \alpha}(v)|_{F'}$ as the function on the $(n-j)$-dimensional affine space $\mathbb R_F^{n-j}$. Then we extend the polynomial $D^{2m-j-|\alpha|}_{F', \alpha}(v)|_{F'}$ to $\mathbb R^n$. For any $\boldsymbol x\in \mathbb R^n$, let $\boldsymbol x_F^P$ be the projection of $\boldsymbol x$ on $\mathbb R_F^{n-j}$. Define
$$
E_K(D^{2m-j-|\alpha|}_{F', \alpha}(v))(\boldsymbol x):=D^{2m-j-|\alpha|}_{F', \alpha}(v)(\boldsymbol x_F^P).
$$
Let $\mathbb R_{F'}^{n}:=\{\bs x\in\mathbb R^n:\boldsymbol x_F^P\in F'\}$, and $\phi_{F}$ be a piecewise polynomial defined as
$$
\phi_{F}(\bs x)=\begin{cases}
\frac{1}{\alpha!}h_{K}^{|\alpha|}b_{F'}^{2m}E_K(D^{2m-j-|\alpha|}_{F', \alpha}(v))\prod\limits_{i=1}^j\lambda_{F',i}^{\alpha_i}, & \bs x\in\mathbb R_{F'}^{n}, F'\in\mathcal{F}_F^j(K), \\
0, & \bs x\in \mathbb R^{n}\backslash\bigcup\limits_{F'\in\mathcal{F}_F^j(K)}\mathbb R_{F'}^{n},
\end{cases}
$$
where $\alpha!=\alpha_1!\cdots\alpha_j!$,
then we have
\begin{equation}\label{eq:20181025-2}
\|\phi_{F}\|_{0,K}\lesssim \sum_{F'\in\mathcal{F}_F^j(K)}h_{K}^{|\alpha|+j/2}\|D^{2m-j-|\alpha|}_{F', \alpha}(v)\|_{0,F'},
\end{equation}
\begin{align*}
\frac{\partial^{|\alpha|}\phi_{F}}{\partial\nu_{F'}^{\alpha}}\Bigg|_{F'}&=\frac{1}{\alpha!}h_{K}^{|\alpha|}b_{F'}^{2m}E_K(D^{2m-j-|\alpha|}_{F', \alpha}(v))\prod_{i=1}^j\frac{\partial^{|\alpha|}(\lambda_{F',i}^{\alpha_i})}{\partial\nu_{F'}^{\alpha}}\\
&=b_{F'}^{2m}E_K(D^{2m-j-|\alpha|}_{F', \alpha}(v)).
\end{align*}
Hence
\begin{equation}\label{eq:20181025-1}
\|D^{2m-j-|\alpha|}_{F', \alpha}(v)\|_{0,F'}^2\eqsim\Bigg(D^{2m-j-|\alpha|}_{F', \alpha}(v), \frac{\partial^{|\alpha|}\phi_{F}}{\partial\nu_{F'}^{\alpha}}\Bigg)_{F'}.
\end{equation}
For each $e\in\mathcal F^{\ell}(K)$ with $\ell=j+1,\cdots,m$, it follows from the fact $b_{F'}|_e=0$ that
$$
\frac{\partial^{|\beta|}\phi_{F}}{\partial\nu_{e}^{\beta}}\bigg|_e=0\quad\forall~\beta\in A_{\ell} \;\textrm{ with }\; |\beta|\leq m-\ell.
$$
Similarly we have for each $e\in\mathcal F^{j}(K)\backslash\mathcal{F}_F^j(K)$
$$
\frac{\partial^{|\beta|}\phi_{F}}{\partial\nu_{e}^{\beta}}\bigg|_e=0\quad\forall~\beta\in A_{j} \;\textrm{ with }\; |\beta|\leq m-j.
$$
For any $\beta\in A_{j}$, $|\beta|<|\alpha|$, since $\frac{\partial^{|\beta|}}{\partial\nu_{F'}^{\beta}}\bigg(\prod\limits_{i=1}^j\lambda_{F',i}^{\alpha_i}\bigg)\bigg|_{F'}=0$, it yields
$\displaystyle
\frac{\partial^{|\beta|}\phi_{F}}{\partial\nu_{F'}^{\beta}}\bigg|_{F'}=0.
$
For any $\beta\in A_{j}$, $|\beta|=|\alpha|$, but $\beta\neq\alpha$, noting that $\frac{\partial\lambda_{F',i}}{\partial\nu_{F',\ell}}=0$ for $i\neq\ell$, we also have
$\displaystyle
\frac{\partial^{|\beta|}\phi_{F}}{\partial\nu_{F'}^{\beta}}\bigg|_{F'}=0.
$
Based on the previous discussion, we obtain from \eqref{eq:20181025-1}, the generalized Green's identity \eqref{eq:HmGreen} and the density argument
\begin{align*}
\sum_{F'\in\mathcal{F}_F^j(K)}\|D^{2m-j-|\alpha|}_{F', \alpha}(v)\|_{0,F'}^2&\simeq(\nabla^mv, \nabla^m\phi_F)_K-((-\Delta)^mv, \phi_F)_K \\
&\quad- \sum_{\ell=1}^{j-1}\sum_{e\in\mathcal F^{\ell}(K)}\sum_{\beta\in A_{\ell}\atop|\beta|\leq m-\ell}\Big (D^{2m-\ell-|\beta|}_{e, \beta}(v), \frac{\partial^{|\beta|}\phi_F}{\partial\nu_{e}^{\beta}}\Big )_e\\
&\quad- \sum_{F'\in\mathcal{F}_F^j(K)}\sum_{\beta\in A_{j}\atop|\alpha|<|\beta|\leq m-j}\Big (D^{2m-j-|\beta|}_{F', \beta}(v), \frac{\partial^{|\beta|}\phi_F}{\partial\nu_{F'}^{\beta}}\Big )_{F'}.
\end{align*}
Employing the Cauchy-Schwarz inequality and the inverse inequality for polynomials, it follows
\begin{align*}
&\sum_{F'\in\mathcal{F}_F^j(K)}\|D^{2m-j-|\alpha|}_{F', \alpha}(v)\|_{0,F'}^2\\
\lesssim &h_K^{-m}\|\nabla^mv\|_{0,K}\|\phi_F\|_{0,K} + \|(-\Delta)^mv\|_{0,K}\|\phi_F\|_{0,K}\\
&+ \sum_{\ell=1}^{j-1}\sum_{e\in\mathcal F^{\ell}(K)}\sum_{\beta\in A_{\ell}\atop|\beta|\leq m-\ell}h_K^{-|\beta|-\ell/2}\Big \|D^{2m-\ell-|\beta|}_{e, \beta}(v)\Big\|_{0,e}\|\phi_F\|_{0,K} \\
&+ \sum_{F'\in\mathcal{F}_F^j(K)}\sum_{\beta\in A_{j}\atop|\alpha|<|\beta|\leq m-j}h_K^{-|\beta|-j/2}\Big \|D^{2m-j-|\beta|}_{F', \beta}(v)\Big\|_{0,F'}\|\phi_F\|_{0,K},
\end{align*}
which combined with \eqref{eq:20181025-2} implies \eqref{eq:inverseeq2}.
\end{proof}

\begin{lemma}
For any $K\in\mathcal T_h$, it holds
\begin{equation}\label{eq:20190605-1}
((-\Delta)^mv, v)_K\lesssim h_K^{m}\|(-\Delta)^mv\|_{0,K}S_K^{1/2}(v,v)\quad\forall~v\in\ker(\Pi^K).
\end{equation}
\end{lemma}
\begin{proof}
If $m\leq k\leq 2m-1$, by the definition of $W_k(K)=V_k(K)$, $(-\Delta)^mv=0$, thus \eqref{eq:20190605-1} is obvious.
Now let us prove \eqref{eq:20190605-1} for $k\geq 2m$.
When $2m\leq k<3m-1$,
since $v\in\ker(\Pi^K)$, it follows from \eqref{eq:20190605}
\begin{align*}
((-\Delta)^mv, v)_K&=((-\Delta)^mv, Q_{m-1}^Kv)_K\\
&=((-\Delta)^mv, Q_{k-2m}^Kv)_K=\left(Q_{k-2m}^K((-\Delta)^mv), v\right)_K.
\end{align*}
If $k\geq 3m-1$, then $W_k(K)=V_k(K)$, and we also have
\[
((-\Delta)^mv, v)_K=\left(Q_{k-2m}^K((-\Delta)^mv), v\right)_K.
\]
Therefore, to derive \eqref{eq:20190605-1} for $k\geq 2m$, it is sufficient to prove
\begin{equation}\label{eq:20190605-2}
\left(Q_{k-2m}^K((-\Delta)^mv), v\right)_K\lesssim h_K^{m}\|(-\Delta)^mv\|_{0,K}S_K^{1/2}(v,v)\quad\forall~v\in\ker(\Pi^K).
\end{equation}
Let $N$ be the dimension of the space $\mathbb P_{k-2m}(K)$. Then there exist constants $c_i$, $i=1,\cdots, N$ such that
\[
Q_{k-2m}^K((-\Delta)^mv)=\sum_{i=1}^{N}c_iq_i,
\]
where $\mathbb M_{k-2m}(K):=\{q_1, \cdots, q_{N}\}$, thus
$$
\left(Q_{k-2m}^K((-\Delta)^mv), v\right)_K=|K|\sum_{i=1}^{N}c_i\chi_i(v).
$$
Applying the norm equivalence on the polynomial space $\mathbb P_{k-2m}(K)$, cf. \eqref{eq:polynorm}, we get
$$
\|Q_{k-2m}^K((-\Delta)^mv)\|_{0,K}^2\eqsim h_K^{n}\sum_{i=1}^{N}c_i^2.
$$
Hence
\begin{align*}
\left(Q_{k-2m}^K((-\Delta)^mv), v\right)_K&\lesssim h_K^{n/2}\|Q_{k-2m}^K((-\Delta)^mv)\|_{0,K}\left (\sum_{i=1}^{N}\chi_i^2(v)\right )^{1/2} \\
&\lesssim h_K^{m}\|Q_{k-2m}^K((-\Delta)^mv)\|_{0,K}S_K^{1/2}(v,v),
\end{align*}
which implies \eqref{eq:20190605-2}.
\end{proof}

\begin{lemma}
For any $K\in\mathcal T_h$, it holds
\begin{equation}\label{eq:SKequiv1}
\|\nabla^m v\|_{0,K}^2\lesssim S_K(v,v)\quad\forall~v\in\ker(\Pi^K).
\end{equation}
\end{lemma}
\begin{proof}
By the generalized Green's identity \eqref{eq:HmGreen},
\begin{equation}\label{eq:20181025-4}
\|\nabla^m v\|_{0,K}^2 = ((-\Delta)^mv, v)_K + \sum_{j=1}^m\sum_{F\in\mathcal F^j(K)}\sum_{\alpha\in A_{j}\atop|\alpha|\leq m-j}\Big (D^{2m-j-|\alpha|}_{F, \alpha}(v), \frac{\partial^{|\alpha|}v}{\partial\nu_{F}^{\alpha}}\Big )_F.\!
\end{equation}
Since $v\in W_k(K)$, we have $D^{2m-j-|\alpha|}_{F, \alpha}(v)|_F\in\mathbb P_{k-(2m-j-|\alpha|)}(F)$ for any $F\in\mathcal F^{j}(K)$. Let $N_F$ be the dimension of the space $\mathbb P_{k-(2m-j-|\alpha|)}(F)$. Then there exist constants $c_i$, $i=1,\cdots, N_F$ such that
$$
\Big (D^{2m-j-|\alpha|}_{F, \alpha}(v), \frac{\partial^{|\alpha|}v}{\partial\nu_{F}^{\alpha}}\Big )_F=h_K^{n-j-|\alpha|}\sum_{i=1}^{N_F}c_i\chi_i(v).
$$
Applying the norm equivalence on the polynomial space $\mathbb P_{k-(2m-j-|\alpha|)}(F)$, cf. \eqref{eq:polynorm}, we get
$$
\|D^{2m-j-|\alpha|}_{F, \alpha}(v)\|_{0,F}^2\eqsim h_K^{n-j}\sum_{i=1}^{N_F}c_i^2.
$$
Hence
\begin{align*}
\Big (D^{2m-j-|\alpha|}_{F, \alpha}(v), \frac{\partial^{|\alpha|}v}{\partial\nu_{F}^{\alpha}}\Big )_F&\lesssim h_K^{(n-j)/2-|\alpha|}\|D^{2m-j-|\alpha|}_{F, \alpha}(v)\|_{0,F}\left (\sum_{i=1}^{N_F}\chi_i^2(v)\right )^{1/2} \\
&\lesssim h_K^{m-|\alpha|-j/2}\|D^{2m-j-|\alpha|}_{F, \alpha}(v)\|_{0,F}S_K^{1/2}(v,v).
\end{align*}
Applying \eqref{eq:inverseeq2} recursively, it follows
\begin{align*}
\Big (D^{2m-j-|\alpha|}_{F, \alpha}(v), \frac{\partial^{|\alpha|}v}{\partial\nu_{F}^{\alpha}}\Big )_F
&\lesssim (\|\nabla^mv\|_{0,K} + h_K^{m}\|(-\Delta)^mv\|_{0,K})S_K^{1/2}(v,v).
\end{align*}
Then we derive from  \eqref{eq:20181025-4}, \eqref{eq:20190605-1} and \eqref{eq:inverseeq1}
\[
\|\nabla^m v\|_{0,K}^2 \lesssim (\|\nabla^mv\|_{0,K} + h_K^{m}\|(-\Delta)^mv\|_{0,K})S_K^{1/2}(v,v)\lesssim \|\nabla^m v\|_{0,K}S_K^{1/2}(v,v),
\]
which induce \eqref{eq:SKequiv1}.
\end{proof}

We then prove another side of the norm equivalence~\eqref{eq:SKequiv}.
\begin{lemma}
For any $K\in\mathcal T_h$ and nonnegative integer $s\leq m$, we have the local Poincar\'e inequality
\begin{equation}\label{eq:poincare}
\sum_{j=0}^{m-s}\sum_{F\in\mathcal F^j(K)}h_K^{s+j/2}\|\nabla^s v\|_{0,F}\lesssim h_K^m\|\nabla^m v\|_{0,K}\quad\forall~v\in \ker(\Pi^K).
\end{equation}
\end{lemma}
\begin{proof}
It is sufficient to prove
\begin{equation}\label{eq:20181103}
\sum_{j=0}^{m-s}\sum_{F\in\mathcal F^j(K)}h_K^{s+j/2}\|\nabla^s v\|_{0,F}\lesssim \sum_{\ell=0}^{m-s-1}\sum_{e\in\mathcal F^{\ell}(K)}h_K^{s+1+\ell/2}\|\nabla^{s+1} v\|_{0,e},
\end{equation}
for $s=0,1,\cdots,m-1$.
Thanks to \eqref{eq:H2projlocal2}, it follows
\begin{align*}
&~h_K^{j/2}\|\nabla^s v\|_{0,F}=h_K^{j/2}\bigg\|\nabla^s v-\frac{1}{\#\mathcal F^{m-s}(K)}\sum_{e\in\mathcal F^{m-s}(K)}Q_0^{e}(\nabla^{s}v)\bigg\|_{0,F} \\
\lesssim& ~h_K^{j/2}\sum_{e\in\mathcal F^{m-s}(K)}\big\|\nabla^s v-Q_0^{e}(\nabla^{s}v)\big\|_{0,F} \\
= &~h_K^{j/2}\sum_{e\in\mathcal F^{m-s}(K)}\big\|\nabla^s v-Q_0^{K}(\nabla^{s}v)-Q_0^{e}(\nabla^{s}v-Q_0^{K}(\nabla^{s}v))\big\|_{0,F} \\
\lesssim &~h_K^{j/2}\big\|\nabla^s v-Q_0^{K}(\nabla^{s}v)\big\|_{0,F}+\sum_{e\in\mathcal F^{m-s}(K)}h_K^{(m-s)/2}\big\|Q_0^{e}(\nabla^{s}v-Q_0^{K}(\nabla^{s}v))\big\|_{0,e} \\
\leq &~h_K^{j/2}\big\|\nabla^s v-Q_0^{K}(\nabla^{s}v)\big\|_{0,F}+\sum_{e\in\mathcal F^{m-s}(K)}h_K^{(m-s)/2}\big\|\nabla^{s}v-Q_0^{K}(\nabla^{s}v)\big\|_{0,e}.
\end{align*}
On the other hand, applying the trace inequality \eqref{L2trace} recursively, we get from \eqref{eq:PKerror}
\begin{align*}
&~h_K^{j/2}\big\|\nabla^s v-Q_0^{K}(\nabla^{s}v)\big\|_{0,F}+\sum_{e\in\mathcal F^{m-s}(K)}h_K^{(m-s)/2}\big\|\nabla^{s}v-Q_0^{K}(\nabla^{s}v)\big\|_{0,e} \\
\lesssim &~\big\|\nabla^s v-Q_0^{K}(\nabla^{s}v)\big\|_{0,K}+\sum_{\ell=0}^{m-s-1}\sum_{e\in\mathcal F^{\ell}(K)}h_K^{1+\ell/2}\|\nabla^{s+1} v\|_{0,e} \\
\lesssim &~\sum_{\ell=0}^{m-s-1}\sum_{e\in\mathcal F^{\ell}(K)}h_K^{1+\ell/2}\|\nabla^{s+1} v\|_{0,e}.
\end{align*}
Combining the last two inequalities yields
\[
h_K^{j/2}\|\nabla^s v\|_{0,F}\lesssim \sum_{\ell=0}^{m-s-1}\sum_{e\in\mathcal F^{\ell}(K)}h_K^{1+\ell/2}\|\nabla^{s+1} v\|_{0,e},
\]
which indicates \eqref{eq:20181103}. Thus the Poincar\'e inequality \eqref{eq:poincare} holds.
\end{proof}

\begin{lemma}
For any $K\in\mathcal T_h$, it holds
\begin{equation}\label{eq:SKequiv2}
S_K(v,v)\lesssim \|\nabla^m v\|_{0,K}^2\quad\forall~v\in \ker(\Pi^K).
\end{equation}
\end{lemma}
\begin{proof}
Due to the definition of the degrees of freedom, we have
\begin{align}
S_K(v,v)&=h_K^{n-2m}\sum_{i=1}^{N_K}\chi_i^2(v)\notag\\
&\lesssim\sum_{j=0}^m\sum_{F\in\mathcal F^j(K)}\sum_{\alpha\in A_{j}\atop|\alpha|\leq m-j}h_K^{2|\alpha|-2m+j}\Big\|Q^F_{k-(2m-j-|\alpha|)}\Big(\frac{\partial^{|\alpha|}v}{\partial\nu_{F}^{\alpha}}\Big)\Big\|_{0,F}^2\notag\\
&\leq\sum_{j=0}^m\sum_{F\in\mathcal F^j(K)}\sum_{\alpha\in A_{j}\atop|\alpha|\leq m-j}h_K^{2|\alpha|-2m+j}\Big\|\frac{\partial^{|\alpha|}v}{\partial\nu_{F}^{\alpha}}\Big\|_{0,F}^2\notag\\
&\leq\sum_{j=0}^m\sum_{F\in\mathcal F^j(K)}\sum_{\alpha\in A_{j}\atop|\alpha|\leq m-j}h_K^{2|\alpha|-2m+j}\|\nabla^{|\alpha|}v\|_{0,F}^2,\notag
\end{align}
which together with the Poincar\'e inequality \eqref{eq:poincare} implies \eqref{eq:SKequiv2}.
\end{proof}

At last, combining \eqref{eq:SKequiv1} and \eqref{eq:SKequiv2} gives the norm equivalence \eqref{eq:SKequiv}, cf. Theorem \ref{th:normequivalence}.

\section{Examples of Green's Formula}\label{sec:appendixB}

Take $K\in\mathcal T_h$.
The explicit expression of \eqref{eq:HmGreen} for $m=1$ with $n\geq m$ is no more than \eqref{eq:H1Green}, i.e.,
\begin{equation*}
(\nabla u, \nabla v)_K=-(\Delta u, v)_K +\sum_{F\in \mathcal F^{1}(K)}(\frac{\partial u}{\partial\nu_{K,F}}, v)_F\quad\forall~u\in H^2(K),\; v\in H^1(K).
\end{equation*}
And the explicit expression of \eqref{eq:HmGreen} for $m=2$ with $n\geq m$ is exactly \eqref{eq:H2Green}, i.e.,
for any $u\in H^4(K)$ and $v\in H^2(K)$, it holds
\begin{align*}
(\nabla^2u, \nabla^2v)_K &= (\Delta^2u, v)_K + \sum_{F\in\mathcal F^{1}(K)}\Big [(M_{\nu\nu}(u), \frac{\partial v}{\partial\nu_{F,1}})_{F} - (Q_{\nu}(u), v)_{F}\Big ] \notag \\
&\qquad+\sum_{e\in\mathcal F^{2}(K)}\sum_{F\in\mathcal F^{1}(K)\cap \partial^{-1}e}(\nu_{F,e}^{\intercal}M_{\nu t}(u), v)_{e}. 
\end{align*}

When $m=n=3$, the explicit expression of \eqref{eq:HmGreen} is that for any $u\in H^{6}(K)$ and $v\in H^3(K)$,
\begin{align*}
&\,(\nabla^3u, \nabla^3v)_K+(\Delta^3u, v)_K\\
=&\sum_{F\in\mathcal F^1(K)}\left(\nu_{K,F}^{\intercal}\div^2(\nabla^3u)+\div_F(\div(\nabla^3u)\nu_{K,F})+\div_F\div_F((\nabla^3u)\nu_{K,F}), v\right)_F \\
&- \sum_{F\in\mathcal F^1(K)}\left(\nu_{F,1}^{\intercal}(\div_F((\nabla^3u)\nu_{K,F}))+\div_F(((\nabla^3u)\nu_{K,F})\nu_{F,1}), \frac{\partial v}{\partial\nu_{F,1}}\right)_F \\
&- \sum_{F\in\mathcal F^1(K)}\left(\nu_{F,1}^{\intercal}\div(\nabla^3u)\nu_{K,F}, \frac{\partial v}{\partial\nu_{F,1}}\right)_F  \\
&+ \sum_{F\in\mathcal F^1(K)}\left(\nu_{F,1}^{\intercal}((\nabla^3u)\nu_{K,F})\nu_{F,1}, \frac{\partial^2 v}{\partial\nu_{F,1}^2}\right)_F \\
&- \sum_{F\in\mathcal F^1(K)}\sum_{e\in\mathcal F^1(F)}\left(\nu_{F,e}^{\intercal}(\div_F((\nabla^3u)\nu_{K,F}))+\div_e(((\nabla^3u)\nu_{K,F})\nu_{F,e}), v\right)_e \\
&- \sum_{F\in\mathcal F^1(K)}\sum_{e\in\mathcal F^1(F)}(\nu_{F,e}^{\intercal}\div(\nabla^3u)\nu_{K,F}, v)_e \\
&+\sum_{F\in\mathcal F^1(K)}\sum_{e\in\mathcal F^1(F)}\sum_{i=1}^2\left(\left(\nu_{e,i}+ (\nu_{e,i}^{\intercal}\nu_{F,1})\nu_{F,1}^{\intercal}\right)((\nabla^3u)\nu_{K,F})\nu_{F,e}, \frac{\partial v}{\partial\nu_{e,i}}\right)_e \\
&+\sum_{F\in\mathcal F^1(K)}\sum_{e\in\mathcal F^1(F)}\sum_{\delta\in\mathcal F^1(e)}\left(\nu_{e,\delta}^{\intercal}((\nabla^3u)\nu_{K,F})\nu_{F,e}\right)(\delta)v(\delta).
\end{align*}
Consider the lowest order case $k=m=3$. The last identity will be reduced to
\begin{align*}
&\,(\nabla^3v, \nabla^3q)_K\\
=& \sum_{F\in\mathcal F^1(K)}\left(\frac{\partial^2 v}{\partial\nu_{F,1}^2}, \nu_{F,1}^{\intercal}((\nabla^3q)\nu_{K,F})\nu_{F,1}\right)_F \\
&+\sum_{F\in\mathcal F^1(K)}\sum_{e\in\mathcal F^1(F)}\sum_{i=1}^2\left(\frac{\partial v}{\partial\nu_{e,i}}, \left(\nu_{e,i}+ (\nu_{e,i}^{\intercal}\nu_{F,1})\nu_{F,1}^{\intercal}\right)((\nabla^3q)\nu_{K,F})\nu_{F,e}\right)_e \\
&+\sum_{F\in\mathcal F^1(K)}\sum_{e\in\mathcal F^1(F)}\sum_{\delta\in\mathcal F^1(e)}v(\delta)\left(\nu_{e,\delta}^{\intercal}((\nabla^3q)\nu_{K,F})\nu_{F,e}\right)(\delta)
\end{align*}
for any $v\in H^3(K)$ and $q\in \mathbb P_3(K)$, which will be used to compute the projector $\Pi^K: H^3(K)\to\mathbb P_3(K)$. And the degrees of freedom are
\[
\left(\frac{\partial^{2}v}{\partial\nu_{F, 1}^{2}}, 1\right)_F,\quad \left(\frac{\partial v}{\partial\nu_{e, 1}}, 1\right)_e,\quad \left(\frac{\partial v}{\partial\nu_{e, 2}}, 1\right)_e,\quad v(\delta)
\]
on each $F\in\mathcal F^{1}(K)$, $e\in\mathcal F^{2}(K)$, and $\delta\in\mathcal F^{3}(K)$.


\end{document}